\documentclass[12pt,etds]{amsart}
\usepackage{amsmath}
\usepackage{amssymb}
\usepackage{amsfonts}
\usepackage{amsthm}
\usepackage{enumerate}
\usepackage{tikz}
\usetikzlibrary{matrix}

\usepackage{pb-diagram}

\newtheorem{theorem}{Theorem}[section]
\newtheorem{lemma}[theorem]{Lemma}
\newtheorem{prop}[theorem]{Proposition}
\newtheorem{cor}[theorem]{Corollary}
\theoremstyle{remark}
\newtheorem{definition}[theorem]{Definition}

\newtheorem{remark}[theorem]{Remark}

\def\i{\iota}

\def\N{{\mathbb N}}

\def\C{{\mathbb C}}

\def\R{{\mathbb R}}
\def\TT{{\mathbb T}}

\def\Z{{\mathbb Z}}

\def\A{{\mathcal{A}}}
\def\B{{\mathcal{B}}}
\def\D{{\mathcal{D}}}

\def\K{{\mathcal{K}}}

\def\T{{\mathcal{T}}}
\def\I{{\mathcal{I}}}
\def\J{{\mathcal{J}}}
\def\L{{\mathcal{L}}}

\def\c{{\bf c}}

\newcommand{\id}{\operatorname{id}}

\newcommand{\piso}{\operatorname{piso}}
\newcommand{\iso}{\operatorname{iso}}

\newcommand{\End}{\operatorname{End}}

\newcommand{\Prim}{\operatorname{Prim}}

\newcommand{\whitesquare}{\hfill $\whitesquare$\newline\vspace{0.4cm}}

\def\newspan{\operatorname{span}}

\numberwithin{equation}{section}

\begin{document}

\title[The composition series of ideals of the partial-isometric crossed product]
{The composition series of ideals of the partial-isometric crossed product by semigroup of endomorphisms}

\author[Sriwulan Adji]{Sriwulan Adji}
\email{sriwulan.adji@gmail.com}

\author[Saeid Zahmatkesh]{Saeid Zahmatkesh}
\email{zahmatkesh.s@gmail.com}



\subjclass[2010]{Primary 46L55}
\keywords{$C^*$-algebra, endomorphism, semigroup, partial isometry, crossed product, primitive ideal, hull-kernel closure}

\begin{abstract}
Let $\Gamma^{+}$ be the positive cone in a totally ordered abelian group $\Gamma$,
and $\alpha$ an action of $\Gamma^{+}$ by extendible endomorphisms of a $C^{\ast}$-algebra $A$.
Suppose $I$ is an extendible $\alpha$-invariant ideal of $A$.
We prove that the partial-isometric crossed product ${\mathcal I}:=I\times_{\alpha}^{\piso}\Gamma^{+}$ embeds naturally as an ideal of
$A\times_{\alpha}^{\piso}\Gamma^{+}$, such that the quotient is the partial-isometric crossed product of the quotient algebra.
We claim that this ideal ${\mathcal I}$ together with the kernel of a natural homomorphism
$\phi: A\times_{\alpha}^{\piso}\Gamma^{+}\rightarrow A\times_{\alpha}^{\iso}\Gamma^{+}$
gives a composition series of ideals of $A\times_{\alpha}^{\piso}\Gamma^{+}$ studied by Lindiarni and Raeburn.
\end{abstract}
\maketitle

\section{Introduction}
Let $(A,\Gamma^{+},\alpha)$ be a dynamical system consisting of the positive cone $\Gamma^{+}$ in a totally ordered abelian group $\Gamma$, and an action
$\alpha:\Gamma^{+}\rightarrow \End A$ of $\Gamma^{+}$ by extendible endomorphisms of a $C^*$-algebra $A$.
A covariant representation of the system $(A,\Gamma^{+},\alpha)$ is defined for which the semigroup of endomorphisms $\{\alpha_{s}: s\in\Gamma^{+}\}$ are implemented by partial isometries, and then the associated partial-isometric crossed product $C^*$-algebra $A\times_{\alpha}^{\piso}\Gamma^{+}$, generated by a universal covariant representation, is characterized by the property that its nondegenerate representations are in a bijective correspondence with covariant representations of the system. This generalizes the covariant isometric representation theory: the theory that uses isometries to represent the semigroup of endomorphisms in a covariant representation of the system. We denoted by $A\times_{\alpha}^{\iso}\Gamma^{+}$ for the corresponding isometric crossed product.

Suppose $I$ is an extendible $\alpha$-invariant ideal of $A$, then $a+I \mapsto \alpha_{x}(a)+I$ defines an action of $\Gamma^{+}$
by extendible endomorphisms of the quotient algebra $A/I$.
It is well-known that the isometric crossed product $I\times_{\alpha}^{\iso}\Gamma^{+}$ sits naturally as an ideal in $A\times_{\alpha}^{\iso}\Gamma^{+}$ such that
$(A\times_{\alpha}^{\iso}\Gamma^{+})/ (I\times_{\alpha}^{\iso}\Gamma^{+})\simeq A/I\times_{\alpha}^{\iso}\Gamma^{+}$.
We show that this result is valid for the partial-isometric crossed product.

Moreover if $\phi: A\times_{\alpha}^{\piso}\Gamma^{+}\rightarrow A\times_{\alpha}^{\iso}\Gamma^{+}$ is the natural homomorphism given by the canonical
universal covariant isometric representation of $(A,\Gamma^{+},\alpha)$ in $A\times_{\alpha}^{\iso}\Gamma^{+}$, then $\ker\phi$ together with
the ideal $I\times_{\alpha}^{\piso}\Gamma^{+}$  give a composition series of ideals of $A\times_{\alpha}^{\piso}\Gamma^{+}$, from which we
recover the structure theorems of \cite{LR}.
Let us now consider the framework of \cite{LR}. A system that consists of the $C^{*}$-subalgebra $A:=B_{\Gamma^{+}}$ of $\ell^{\infty}(\Gamma^{+})$ spanned by the functions $1_{s}$ satisfying
\[ 1_{s}(t)=\left\{ \begin{array}{ll} 1 & \mbox{ if } t \ge s \\ 0 & \mbox{ otherwise, } \end{array} \right. \]
and the action $\tau:\Gamma^{+}\rightarrow \End B_{\Gamma^{+}}$ given by the translation on $\ell^{\infty}(\Gamma^{+})$.
We choose an extendible $\tau$-invariant ideal $I$ to be the subalgebra $B_{\Gamma^{+},\infty}$ spanned by $\{1_{x}-1_{y} : x<y \in \Gamma^{+}\}$.
Then the composition series of ideals of $B_{\Gamma^{+}}\times_{\tau}^{\piso}\Gamma^{+}$, that is given by
the two ideals $\ker\phi$ and $B_{\Gamma^{+},\infty}\times_{\tau}^{\piso}\Gamma^{+}$, produces
the large commutative diagram in \cite[Theorem 5.6]{LR}.
This result shows that the commutative diagram in \cite[Theorem 5.6]{LR} exists for any totally ordered abelian subgroup (not only for subgroups of $\R$), and that we understand clearly where the diagram comes from.

Next, if we consider a specific semigroup $\Gamma^{+}$ such as the additive semigroup $\N$ in the group of integers $\Z$, then the large commutative diagram gives a clearer information about the ideals structure of $\c\times_{\tau}^{\piso}\N$.
We can identify that the left-hand and top exact sequences in diagram \cite[Theorem 5.6]{LR} are indeed equivalent to the extension of the algebra $\K(\ell^{2}(\N,\c_{0}))$ of compact operators on the Hilbert module $\ell^{2}(\N,\c_{0})$ by $\K(\ell^{2}(\N))$ provided by the algebra $\K(\ell^{2}(\N,\c))$ of compact operators on $\ell^{2}(\N,\c)$.
Moreover it is known that $\Prim \K(\ell^{2}(\N,\c))\simeq \Prim (\K(\ell^{2}(\N)) \otimes \c) \simeq \Prim \c$ is homeomorphic to $\N\cup\infty$.
Together with a knowledge about the primitive ideal space of the Toeplitz $C^*$-algebra generated by the unilateral shift, our theorem on the composition series of ideals of $\c\times_{\tau}^{\piso}\N$ provides a complete description of the topology on the primitive ideal space of $\c\times_{\tau}^{\piso}\N$.

We begin with a section containing background material about the partial-isometric crossed product by semigroups of extendible endomorphisms.
In Section 3, we prove the existence of a short exact sequence of partial-isometric crossed products,
which generalizes \cite[ Theorem 2.2]{Adji-Abbas} of the semigroup $\N$.
Then we consider this and the other natural exact sequence described earlier in \cite{AZ},
to get the composition series of ideals in $A\times^{\piso}_{\alpha}\Gamma^{+}$.

We proceed to Section 4 by applying our results in Section 3 to the distinguished system $(B_{\Gamma^{+}},\Gamma^{+},\tau)$ and the extendible $\tau$-invariant ideal
$B_{\Gamma^{+},\infty}$ of $B_{\Gamma^{+}}$.
It can be seen from our Proposition \ref{BGamma} that the large commutative diagram of \cite[Theorem 5.6]{LR} remains valid for any subgroup $\Gamma$
of a totally ordered abelian group.
Finally in the last section we describe the topology of primitive ideal space of $\c\times^{\piso}_{\tau}\N$ by using this large diagram.

\section{Preliminaries}
A bounded operator $V$ on a Hilbert space $H$ is called an isometry if $\|V(h)\|=\|h\|$ for all $h\in H$, which is equivalent to $V^{*}V=1$.
A bounded operator $V$ on a Hilbert space $H$ is called a \emph{partial isometry} if it is isometry on $(\ker V)^{\perp}$.
This is equivalent to $VV^{*}V=V$.
If $V$ is a partial isometry then so is the adjoint $V^*$, where as for an isometry $V$, the adjoint $V^{*}$ may not be an isometry
unless $V$ is unitary.
Associated to a partial isometry $V$, there are two orthogonal projections $V^{*}V$ and $VV^{*}$ on the initial space $(\ker V)^{\perp}$
and on the range $VH$ respectively.
In a $C^*$-algebra $A$, an element $v\in A$ is called an isometry if $v^{*}v=1$ and a partial isometry if $vv^{*}v=v$.

An \emph{isometric representation} of $\Gamma^{+}$ on a Hilbert space $H$ is a map $S:\Gamma^{+}\rightarrow B(H)$ which satisfies
$S_{x}:=S(x)$ is an isometry, and $S_{x+y}=S_{x}S_{y}$ for all $x,y\in\Gamma^{+}$.
So an isometric representation of $\N$ is determined by a single isometry $S_{1}$.
Similarly a \emph{partial-isometric representation} of $\Gamma^{+}$ on a Hilbert space $H$ is a map $V:\Gamma^{+}\rightarrow B(H)$ which satisfies
$V_{x}:=V(x)$ is a partial isometry, and $V_{x+y}=V_{x}V_{y}$ for all $x,y\in\Gamma^{+}$.
Note that the product $VW$ of two partial isometries $V$ and $W$ is a partial isometry precisely when $V^{*}V$ commutes with $WW^{*}$ \cite[Proposition 2.1]{LR}.
Thus a partial isometry $V$ is called a \emph{power partial isometry} if $V^{n}$ is a partial isometry for every $n\in\N$, so a
partial-isometric  representation of $\N$ is determined by a single power partial isometry $V_{1}$.
If $V$ is a partial-isometric representation of $\Gamma^{+}$, then every $V_{x}V_{x}^{*}$ commutes with $V_{t}V_{t}^{*}$, and so does
$V_{x}^{*}V_{x}$ with $V_{t}^{*}V_{t}$.

Now we consider a dynamical system $(A,\Gamma^{+},\alpha)$ consisting of a $C^*$-algebra $A$, an action $\alpha$ of $\Gamma^{+}$ by endomorphisms of $A$ such that
$\alpha_{0}=\id$. Because we deal with non unital $C^*$-algebras and non unital endomorphisms,
we require every endomorphism $\alpha_{x}$ to be extendible to a strictly continuous endomorphism $\overline{\alpha}_{x}$ on
the multiplier algebra $M(A)$ of $A$. This happens precisely when there exists an approximate identity $(a_{\lambda})$ in $A$  and a projection $p_{\alpha_{x}}\in M(A)$ such that $\alpha_{x}(a_{\lambda})$ converges strictly to $p_{\alpha_{x}}$ in $M(A)$.

\begin{definition}
A \emph{covariant isometric representation} of $(A,\Gamma^{+},\alpha)$ on a Hilbert space $H$ is a pair $(\pi,S)$ of
a nondegenerate representation $\pi:A\rightarrow B(H)$ and an isometric representation of $S:\Gamma^{+}\rightarrow B(H)$ such that
$\pi(\alpha_{x}(a))=S_{x}\pi(a) S_{x}^{*}$ for all $a\in A$ and $x\in\Gamma^{+}$.

An \emph{isometric crossed product} of $(A,\Gamma^{+},\alpha)$ is a triple $(B,j_{A},j_{\Gamma^{+}})$ consisting of a $C^*$-algebra $B$,
a canonical covariant isometric representation $(j_{A},j_{\Gamma^{+}})$ in $M(B)$ which satisfies the following
\begin{itemize}
\item[(i)] for every covariant isometric representation $(\pi,S)$ of $(A,\Gamma^{+},\alpha)$ on a Hilbert space $H$, there exists a nondegenerate representation
$\pi\times S: B\rightarrow B(H)$ such that $(\pi\times S) \circ j_{A}=\pi$ and $(\overline{\pi\times S}) \circ j_{\Gamma^{+}}=S$; and
\item[(ii)] $B$ is generated by $j_{A}(A)\cup j_{\Gamma^{+}}(\Gamma^{+})$, we actually have
\[ B=\overline{\newspan}\{j_{\Gamma^{+}}(x)^{*} j_{A}(a) j_{\Gamma^{+}}(y) : x,y \in\Gamma^{+}, a\in A\}. \]
\end{itemize}
Note that a given system $(A,\Gamma^{+},\alpha)$ could have a covariant isometric representation $(\pi,S)$ only with $\pi=0$.
In this case the isometric crossed product yields no information about the system.
If a system admits a non trivial covariant representation, then the isometric crossed product does exist, and it is unique up to isomorphism:
if there is such a covariant isometric representation
$(t_{A},t_{\Gamma^{+}})$ of $(A,\Gamma^{+},\alpha)$ in a $C^*$-algebra $C$, then there is an isomorphism of $C$ onto $B$ which takes
$(t_{A},t_{\Gamma^{+}})$ into $(j_{A},j_{\Gamma^{+}})$.
Thus we write the isometric crossed product $B$ as $A\times_{\alpha}^{\iso}\Gamma^{+}$.
\end{definition}

The partial-isometric crossed product of $(A,\Gamma^{+},\alpha)$ is defined in a similar fashion involving partial-isometries instead of isometries.
\begin{definition}
A \emph{covariant partial-isometric representation} of $(A,\Gamma^{+},\alpha)$ on a Hilbert space $H$ is a pair $(\pi,S)$ of
a nondegenerate representation $\pi:A\rightarrow B(H)$ and a partial-isometric representation $S:\Gamma^{+}\rightarrow B(H)$ of $\Gamma^{+}$ such that
$\pi(\alpha_{x}(a))=S_{x}\pi(a) S_{x}^{*}$ for all $a\in A$ and $x\in\Gamma^{+}$.
See in the Remark \ref{cov2} that this equation implies $S_{x}^{*}S_{x} \pi(a)=\pi(a) S_{x}^{*}S_{x}$ for $a\in A$ and $x\in\Gamma^{+}$.
Moreover, \cite[Lemma 4.2]{LR} shows that every $(\pi,S)$ extends to a partial-isometric covariant representation $(\overline{\pi},S)$ of
$(M(A),\Gamma^{+},\overline{\alpha})$, and the partial-isometric covariance is equivalent to
$\pi(\alpha_{x}(a))S_{x}=S_{x}\pi(a)$ and $S_{x}S_{x}^{*}=\overline{\pi}(\overline{\alpha}_{x}(1))$ for $a\in A$ and $x\in\Gamma^{+}$.

A \emph{partial-isometric crossed product} of $(A,\Gamma^{+},\alpha)$ is a triple $(B,j_{A},j_{\Gamma^{+}})$ consisting of a $C^*$-algebra $B$,
a canonical covariant partial-isometric representation $(j_{A},j_{\Gamma^{+}})$ in $M(B)$ which satisfies the following
\begin{itemize}
\item[(i)] for every covariant partial-isometric representation $(\pi,S)$ of $(A,\Gamma^{+},\alpha)$ on a Hilbert space $H$,
there exists a nondegenerate representation
$\pi\times S: B\rightarrow B(H)$ such that $(\pi\times S) \circ j_{A}=\pi$ and $(\overline{\pi\times S}) \circ j_{\Gamma^{+}}=S$; and
\item[(ii)] $B$ is generated by $j_{A}(A)\cup j_{\Gamma^{+}}(\Gamma^{+})$, we actually have
\[ B=\overline{\newspan}\{j_{\Gamma^{+}}(x)^{*} j_{A}(a) j_{\Gamma^{+}}(y) : x,y \in\Gamma^{+}, a\in A\}. \]
\end{itemize}
Unlike the theory of isometric crossed product:  every system $(A,\Gamma^{+},\alpha)$ admits a non trivial covariant partial-isometric representation
$(\pi,S)$ with $\pi$ faithful \cite[Example 4.6]{LR}.
In fact \cite[Proposition 4.7]{LR} shows that a canonical covariant partial-isometric representation $(j_{A},j_{\Gamma^{+}})$ of $(A,\Gamma^{+},\alpha)$
exists in the Toeplitz algebra $\T_{X}$ associated to a discrete product system $X$ of Hilbert bimodules over
$\Gamma^{+}$, which (i) and (ii) are fulfilled, and it is universal: if there is such a covariant partial-isometric representation
$(t_{A},t_{\Gamma^{+}})$ of $(A,\Gamma^{+},\alpha)$ in a $C^*$-algebra $C$ that satisfies (i) and (ii), then there is an isomorphism of $C$ onto $B$ which takes
$(t_{A},t_{\Gamma^{+}})$ into $(j_{A},j_{\Gamma^{+}})$.
Thus we write the partial-isometric crossed product $B$ as $A\times_{\alpha}^{\piso}\Gamma^{+}$.
\end{definition}

\begin{remark}\label{cov2}
Our special thanks go to B. Kwasniewski for showing us the proof arguments in this remark.
Assuming $(\pi,S)$ is covariant, then by $C^*$-norm equation  we have
$\|\pi(a)S_{x}^{*}-S^{*}_{x}\pi(\alpha_{x}(a))\|=0$, therefore $\pi(a)S_{x}^{*}=S^{*}_{x}\pi(\alpha_{x}(a))$ for all $a\in A$ and $x\in\Gamma^{+}$,
which means that $S_{x}\pi(a)=\pi(\alpha_{x}(a))S_{x}$ for all $a\in A$ and $x\in\Gamma^{+}$.
So $S_{x}^{*}S_{x}\pi(a)=S_{x}^{*}\pi(\alpha_{x}(a))S_{x}=(\pi(\alpha_{x}(a^{*}))S_{x})^{*}S_{x}=
(S_{x}\pi(a^{*}))^{*}S_{x}=\pi(a)S_{x}^{*} S_{x}$.
\end{remark}

\section{The short exact sequence of partial-isometric crossed products}
\begin{theorem}\label{ses-pisoG}
Suppose that $(A\times^{\piso}_{\alpha}\Gamma^{+},i_{A},V)$ is the partial-isometric crossed product of a dynamical system
$(A,\Gamma^{+},\alpha)$, and $I$ is an extendible $\alpha$-invariant ideal of $A$.
Then there is a short exact sequence
\begin{equation}\label{diagram-ses-t}
\begin{diagram}\dgARROWLENGTH=0.5\dgARROWLENGTH
\node{0} \arrow{e} \node{I\times^{\piso}_{\alpha}\Gamma^{+}} \arrow{e,t}{\mu}\node{A\times^{\piso}_{\alpha}\Gamma^{+}}
\arrow{e,t}{\gamma}\node{A/I\times^{\piso}_{\tilde{\alpha}}\Gamma^{+}}\arrow{e}\node{0,}
\end{diagram}
\end{equation}
where $\mu$ is an isomorphism of $I\times^{\piso}_{\alpha}\Gamma^{+}$ onto the ideal
\[ \D:=\overline{\newspan} \{V_{x}^{*}i_{A}(i)V_{y}: i\in I, x,y \in \Gamma^{+}\} \text{ of } A\times^{\piso}_{\alpha}\Gamma^{+}. \]
If $q:A\rightarrow A/I$ is the quotient map, $i_{I},W$ denote the maps $I \rightarrow I\times^{\piso}_{\alpha}\Gamma^{+}$,
$W: \Gamma^{+}\rightarrow M(I\times^{\piso}_{\alpha}\Gamma^{+})$, and similarly for $i_{A/I}, U$ the maps
$A/I \rightarrow A/I\times^{\piso}_{\tilde{\alpha}}\Gamma^{+}$, $\Gamma^{+}\rightarrow M(A/I\times^{\piso}_{\tilde{\alpha}}\Gamma^{+})$,
then
\[ \mu\circ i_{I}=i_{A}|_{I}, \quad \overline{\mu}\circ W= V \quad \text{ and } \quad
\gamma\circ i_{A}=i_{A/I}\circ q, \quad \overline{\gamma}\circ V= U. \]
\end{theorem}
\begin{proof}
We make some minor adjustment to the proof of \cite[Theorem 3.1]{Adji1} for partial isometries.
First, to check that $\D$ is indeed an ideal of $A\times^{\piso}_{\alpha}\Gamma^{+}$.
Let $\xi=V_{x}^{*}i_{A}(i)V_{y} \in\D$.
Then $V_{s}^{*}\xi$ is trivially contained in $\D$, and
computations below show that $i_{A}(a)\xi$ and $V_{s}\xi$ are all in $\D$ for $a\in A$ and $s\in \Gamma^{+}$:
\begin{align*}
i_{A}(a)\xi  & = i_{A}(a)V_{x}^{*}i_{A}(i)V_{y}=(V_{x}i_{A}(a^{*}))^{*}i_{A}(i)V_{y} \\
& =(i_{A}(\alpha_{x}(a^{*}))V_{x})^{*}i_{A}(i)V_{y}= V_{x}^{*}i_{A}(\alpha_{x}(a)i)V_{y};
\end{align*}
\begin{align*}
V_{s}\xi = &  V_{s}V_{x}^{*}i_{A}(i)V_{y}=V_{s}(V_{s}^{*}V_{s}V_{x}^{*}V_{x})V_{x}^{*}i_{A}(i)V_{y}
=  V_{s}V_{u}^{*}V_{u} V_{x}^{*} i_{A}(i) V_{y}, \quad u:=\max\{s,x\}\\
= & (V_{s}V_{s}^{*} V_{u-s}^{*})(V_{u-x}V_{x}V_{x}^{*})i_{A}(i)V_{y} = V_{u-s}^{*} (V_{u}V_{u}^{*}V_{u} V_{u}^{*})(V_{u-x}i_{A}(i))V_{y}\\
= & V_{u-s}^{*}V_{u}V_{u}^{*}i_{A}(\alpha_{u-x}(i))V_{u-x}V_{y}
= V_{u-s}^{*} \overline{i}_{A}(\overline{\alpha}_{u}(1))i_{A}(\alpha_{u-x}(i))V_{u-x+y}.
\end{align*}

This ideal $\D$ gives us a nondegenerate homomorphism $\psi:A\times^{\piso}_{\alpha}\Gamma^{+} \rightarrow M(D)$ which satisfies
$\psi(\xi)d=\xi d$ for $\xi \in A\times^{\piso}_{\alpha}\Gamma^{+}$ and $d\in \D$.
Let $j_{I}: I \stackrel{i_{A}}{\longrightarrow} A\times^{\piso}_{\alpha}\Gamma^{+} \stackrel{\psi}{\longrightarrow} M(\D)$, and
$S:\Gamma^{+} \stackrel{V}\longrightarrow M(A\times^{\piso}_{\alpha}\Gamma^{+}) \stackrel{\overline{\psi}}{\longrightarrow} M(\D)$.
We use extendibility of ideal $I$ to show $j_{I}$ is nondegenerate.
Take an approximate identity $(e_{\lambda})$ for $I$, and let $\varphi:A\rightarrow M(I)$ be the homomorphism satisfying $\varphi(a)i=ai$ for $a\in A$ and $i\in I$.
Then $i_{A}(\alpha_{s}(e_{\lambda})i)$ converges in norm to $i_{A}(\overline{\varphi}(\overline{\alpha}_{s}(1_{M(A)}))i)$.
However
\[ i_{A}(\overline{\varphi}(\overline{\alpha}_{s}(1_{M(A)}))i)=\overline{i}_{A}(\overline{\alpha}_{s}(1_{M(A)}))i_{A}(i)=
V_{s}V_{s}^{*}i_{A}(i).\]
So $i_{A}(\alpha_{s}(e_{\lambda})i)$ converges in norm to $V_{s}V_{s}^{*}i_{A}(i)$.
Since  $j_{I}(e_{\lambda})V_{s}^{*}i_{A}(i)V_{t}=V_{s}^{*}i_{A}(\alpha_{s}(e_{\lambda})i)V_{t}$ by covariance, it follows that
$j_{I}(e_{\lambda})V_{s}^{*}i_{A}(i)V_{t}$ converges in norm to $V_{s}^{*}i_{A}(i)V_{t}$. We can similarly show that $V_{s}^{*}i_{A}(i)V_{t}j_{I}(e_{\lambda})$ converges in norm to $V_{s}^{*}i_{A}(i)V_{t}$.
Thus $j_{I}(e_{\lambda})\rightarrow 1_{M(\D)}$ strictly, and hence $j_{I}$ is nondegenerate.

We claim that the triple $(\D,j_{I},S)$ is a partial-isometric crossed product of $(I,\Gamma^{+},\alpha)$.
A routine computations show the covariance of $(j_{I},S)$ for $(I,\Gamma^{+},\alpha)$.
Suppose now $(\pi,T)$ is a covariant representation of $(I,\Gamma^{+},\alpha)$ on a Hilbert space $H$.
Let $\rho: A\stackrel{\varphi}{\longrightarrow} M(I)\stackrel{\overline{\pi}}{\longrightarrow} B(H)$.
Then by extendibility of ideal $I$, that is $\overline{\alpha|_{I}}\circ \varphi=\varphi\circ \alpha$, the pair $(\rho,T)$ is a covariant representation of
$(A,\Gamma^{+},\alpha)$.
The restriction $(\rho\times T)|_{\D}$ to $\D$ of $\rho\times T$ is a nondegenerate representation of $\D$ which satisfies the requirement
$(\rho\times T)|_{\D}\circ j_{I}= \pi$ and $\overline{(\rho\times T)|_{\D}}\circ S= T$.
Thus the triple $(\D,j_{I},S)$ is a partial-isometric crossed product for $(I,\Gamma^{+},\alpha)$, and we have
the homomorphism $\mu=i_{A}|_{I}\times V$.

Next we show the exactness. Let $\Phi$ be a nondegenerate representation of $A\times_{\alpha}^{\piso}\Gamma^{+}$ with kernel $\D$.
Since $I\subset \ker \Phi\circ i_{A}$, we can have a representation $\tilde{\Phi}$ of $A/I$, which
together with $\overline{\Phi}\circ V$ is a covariant partial-isometric representation of $(A/I,\Gamma^{+},\tilde{\alpha})$.
Then $\tilde{\Phi}\times (\overline{\Phi}\circ V)$ lifts to $\Phi$, and therefore $\ker\gamma \subset \ker \Phi=\D$.
\end{proof}

\begin{cor}\label{square-piso-iso}
Let $(A,\Gamma^{+},\alpha)$ be a dynamical system, and $I$ an extendible $\alpha$-invariant ideal of $A$.
Then there is a commutative diagram
\begin{center}
\begin{tikzpicture}
 \matrix(m)[matrix of math nodes,row sep=2em,column sep=2em,minimum width=2em]
 {\empty & 0 & 0 & 0 & \empty\\
  0 & \ker\phi_{I} & I\times^{\piso}_{\alpha}\Gamma^{+} & I\times^{\iso}_{\alpha}\Gamma^{+}  & 0\\
  0 & \ker\phi_{A} & A\times^{\piso}_{\alpha}\Gamma^{+} &   A\times_{\alpha}^{\iso}\Gamma^{+}  & 0\\
  0 & \ker\phi_{A/I} & A/I\times^{\piso}_{\tilde{\alpha}}\Gamma^{+}  & A/I\times_{\tilde{\alpha}}^{\iso}\Gamma^{+} & 0\\
   \empty & 0 & 0 & 0 & \empty\\};
  \path[-stealth]
    (m-1-2) edge node [above] {} (m-2-2)
    (m-1-3) edge node [above] {} (m-2-3)
    (m-1-4) edge node [above] {} (m-2-4)
    (m-2-1) edge node [above] {} (m-2-2)
    (m-2-2) edge node [above] {} (m-2-3) edge node [left] {} (m-3-2)
    (m-2-3) edge node [above] {$\phi_{I}$} (m-2-4) edge node [left] {$\mu$} (m-3-3)
    (m-2-4) edge node [] {} (m-2-5) edge node [left] {$\mu^{\iso}$} (m-3-4)
    (m-3-1) edge node [] {} (m-3-2)
    (m-3-2) edge node [above] {} (m-3-3) edge node [left] {} (m-4-2)
    (m-3-3) edge node [above] {$\phi_{A}$} (m-3-4) edge node [left] {$\gamma$} (m-4-3)
    (m-3-4) edge node [] {} (m-3-5) edge node [left] {$\gamma^{\iso}$} (m-4-4)
    (m-4-1) edge node [] {} (m-4-2)
    (m-4-2) edge node [above] {} (m-4-3) edge node [left] {} (m-5-2)
    (m-4-3) edge node [above] {$\phi_{A/I}$} (m-4-4) edge node [left] {} (m-5-3)
    (m-4-4) edge node [] {} (m-4-5) edge node [left] {} (m-5-4);
\end{tikzpicture}
\end{center}
\end{cor}
\begin{proof}
The three row exact sequences follow from \cite{AZ}, the middle column from Theorem \ref{ses-pisoG} and
the right column exact sequence from \cite{Adji1}.
By inspection on the spanning elements, one can see that  $\mu(\ker\phi_{I})$ is an ideal of $\ker\phi_{A}$ and $\mu^{\iso}\circ \phi_{I}=\phi_{A}\circ \mu$,
thus first and second rows commute.
Then Snake Lemma gives the commutativity of all rows and columns.
\end{proof}

\section{The Example}
We consider a dynamical system $(B_{\Gamma^{+}},\Gamma^{+},\tau)$  consisting of a unital $C^*$-subalgebra $B_{\Gamma^{+}}$ of $\ell^{\infty}(\Gamma^{+})$
spanned by the set $\{1_{s} : s\in\Gamma^{+}\}$ of characteristic functions $1_{s}$ of $\{x\in \Gamma^{+} : x\ge s \}$, the action
$\tau$ of $\Gamma^{+}$ on $B_{\Gamma^{+}}$  is given by  $\tau_{x}(1_{s})=1_{s+x}$.
The ideal $B_{\Gamma^{+},\infty}=\overline{\newspan}\{1_{i}-1_{j} : i<j \in\Gamma^{+}\}$ is an extendible $\tau$-invariant ideal of $B_{\Gamma^{+}}$.
Then we want to show in Proposition \ref{BGamma} that an application of Corollary \ref{square-piso-iso} to the system $(B_{\Gamma^{+}},\Gamma^{+},\tau)$
and the ideal $B_{\Gamma^{+},\infty}$ gives  \cite[Theorem 5.6]{LR}.

The crossed product $B_{\Gamma^{+}}\times_{\tau}^{\iso}\Gamma^{+}$ is a universal $C^*$-algebra generated by
the canonical isometric representation $t$ of $\Gamma^{+}$: every isometric representation $w$ of $\Gamma^{+}$ gives
a covariant isometric representation $(\pi_{w},w)$ of $(B_{\Gamma^{+}},\Gamma^{+},\tau)$.
Suppose $\{\varepsilon_{x} : x\in\Gamma^{+}\}$ is the usual orthonormal basis in $\ell^{2}(\Gamma^{+})$,
and let $T_{s}(\varepsilon_{x})=\varepsilon_{x+s}$ for every $s\in\Gamma^{+}$.
Then $s\mapsto T_{s}$ is an isometric representation of $\Gamma^{+}$, and the Toeplitz algebra $\T(\Gamma)$ is the $C^*$-subalgebra of
$B(\ell^{2}(\Gamma^{+}))$ generated by $\{T_{s} : s\in\Gamma^{+}\}$.
So there exists a representation $\mathfrak{T}:=\pi_{T}\times T$ of $B_{\Gamma^{+}}\times_{\tau}^{\iso}\Gamma^{+}$ on $\ell^{2}(\Gamma^{+})$
such that $\mathfrak{T}(t_{x})=T_{x}$ and $\mathfrak{T}(1_{x})=T_{x}T_{x}^{*}$ for all $x\in\Gamma^{+}$.
This representation is faithful by \cite[Theorem 2.4]{ALNR}.
Thus $B_{\Gamma^{+}}\times_{\tau}^{\iso}\Gamma^{+}$ and the Toeplitz algebra
$\T(\Gamma)=\pi_{T}\times T(B_{\Gamma^{+}}\times_{\tau}^{\iso}\Gamma^{+})$ are isomorphic, and the isomorphism
takes the ideal $B_{\Gamma^{+},\infty}\times_{\tau}^{\iso}\Gamma^{+}$ of $B_{\Gamma^{+}}\times_{\tau}^{\iso}\Gamma^{+}$
onto the commutator ideal ${\mathcal C}_{\Gamma}=\overline{\newspan}\{T_{x}(1-TT^{*})T_{y}^{*} : x,y \in\Gamma^{+}\}$ of $\T(\Gamma)$.

Similarly, the crossed product $B_{\Gamma^{+}}\times_{\tau}^{\piso}\Gamma^{+}$ has
a partial-isometric version of universal property by \cite[Proposition 5.1]{LR}:
every partial-isometric representation $v$ of $\Gamma^{+}$ gives a covariant partial-isometric representation
$(\pi_{v},v)$ of $(B_{\Gamma^{+}},\Gamma^{+},\tau)$ with $\pi_{v}(1_{x})=v_{x}v_{x}^{*}$,
and then $B_{\Gamma^{+}}\times_{\tau}^{\piso}\Gamma^{+}$ is the universal $C^*$-algebra generated by the canonical partial-isometric representation
$v$ of $\Gamma^{+}$.
Now since $x\mapsto T_{x}$ and $x\mapsto T^{*}_{x}$ are partial-isometric
representations of $\Gamma^{+}$ in the Toeplitz algebra $\T(\Gamma)$, there exist (by the universality) a homomorphism
$\varphi_{T}$ and $\varphi_{T^{*}}$ of $B_{\Gamma^{+}}\times_{\tau}^{\piso}\Gamma^{+}$ onto $\T(\Gamma)$.

Next consider the algebra $C(\hat{\Gamma})$ generated by $\{\lambda_{x} : x\in \Gamma\}$ of the evaluation maps $\lambda_{x}(\xi)=\xi(x)$ on $\hat{\Gamma}$.
Let $\psi_{T}$ and $\psi_{T^{*}}$ be the homomorphisms of $\T(\Gamma)$ onto $C(\hat{\Gamma})$ defined by
$\psi_{T}(T_{x})=\lambda_{x}$ and $\psi_{T^{*}}(T_{x})=\lambda_{-x}$.

\begin{prop}\cite[Theorem 5.6]{LR}\label{BGamma} Let $\Gamma^{+}$ be the positive cone in a totally ordered abelian group $\Gamma$.
Then the following commutative diagram exists:
\begin{center}
\begin{equation}\label{diagram0}
\begin{tikzpicture}
  \matrix(m)[matrix of math nodes,row sep=2em,column sep=2em,minimum width=2em]
  {\empty & 0 & 0 & 0 & \empty\\
  0 & \ker\varphi_{T}\cap\ker\varphi_{T^{*}} & \ker\varphi_{T^{*}} & {\mathcal C}_{\Gamma} & 0\\
   0 & \ker\varphi_{T} & B_{\Gamma^{+}}\times_{\tau}^{\piso}\Gamma^{+} & \mathcal{T}(\Gamma) & 0\\
   0 & {\mathcal C}_{\Gamma} &\mathcal{T}(\Gamma) & C(\hat{\Gamma}) & 0\\
   \empty & 0 & 0 & 0 & \empty\\};
  \path[-stealth]
    (m-1-2) edge node [above] {} (m-2-2)
    (m-1-3) edge node [above] {} (m-2-3)
    (m-1-4) edge node [above] {} (m-2-4)
    (m-2-1) edge node [above] {} (m-2-2)
    (m-2-2) edge node [above] {} (m-2-3) edge node [left] {} (m-3-2)
    (m-2-3) edge node [above] {$\varphi_{T}|$} (m-2-4) edge node [left] {} (m-3-3)
    (m-2-4) edge node [] {} (m-2-5) edge node [left] {} (m-3-4)
    (m-3-1) edge node [] {} (m-3-2)
    (m-3-2) edge node [above] {} (m-3-3)  edge node [left] {} (m-4-2)
    (m-3-3) edge node [above] {$\Psi$} (m-4-4) edge node [above] {$\varphi_{T}$} (m-3-4) edge node [left] {$\varphi_{T^*}$} (m-4-3)
    (m-3-4) edge node [] {} (m-3-5) edge node [right] {$\psi_{T^*}$} (m-4-4)
    (m-4-1) edge node [] {} (m-4-2)
    (m-4-2) edge node [above] {} (m-4-3) edge node [left] {} (m-5-2)
    (m-4-3) edge node [above] {$\psi_{T}$} (m-4-4) edge node [left] {} (m-5-3)
    (m-4-4) edge node [] {} (m-4-5) edge node [left] {} (m-5-4);
\end{tikzpicture}
\end{equation}
\end{center}
where $\Psi$ maps each generator $v_{x}\in B_{\Gamma^{+}}\times_{\tau}^{\piso}\Gamma^{+}$ to $\delta_{x}^{*}\in C^{*}(\Gamma)\simeq C(\hat{\Gamma})$.
\end{prop}
\begin{proof}
Apply Corollary \ref{square-piso-iso} to the system $(B_{\Gamma^{+}},\Gamma^{+},\tau)$
and the extendible ideal $B_{\Gamma^{+},\infty}$.
Let $Q^{\piso}:=B_{\Gamma^{+}}/B_{\Gamma^{+},\infty}\times^{\piso}_{\tilde{\tau}}\Gamma^{+}$ and
$Q^{\iso}:=B_{\Gamma^{+}}/B_{\Gamma^{+},\infty}\times^{\iso}_{\tilde{\tau}}\Gamma^{+}$.
Then we have:
\begin{center}
\begin{equation}
\begin{tikzpicture} \label{diagram1}
  \matrix(m)[matrix of math nodes,row sep=2em,column sep=2em,minimum width=2em]
  {\empty & 0 & 0 & 0 & \empty\\
  0 & \ker\phi_{B_{\Gamma^{+},\infty}} & B_{\Gamma^{+},\infty}\times_{\tau}^{\piso}\Gamma^{+} & B_{\Gamma^{+},\infty}\times_{\tau}^{\iso}\Gamma^{+} & 0\\
   0 & \ker\phi_{B_{\Gamma^{+}}} & B_{\Gamma^{+}}\times_{\tau}^{\piso}\Gamma^{+} & B_{\Gamma^{+}}\times_{\tau}^{\iso}\Gamma^{+} & 0\\
   0 & \ker\phi_{Q} & Q^{\piso} &  Q^{\iso} & 0\\
   \empty & 0 & 0 & 0 & \empty\\};
  \path[-stealth]
    (m-1-2) edge node [above] {} (m-2-2)
    (m-1-3) edge node [above] {} (m-2-3)
    (m-1-4) edge node [above] {} (m-2-4)
    (m-2-1) edge node [above] {} (m-2-2)
    (m-2-2) edge node [above] {} (m-2-3) edge node [left] {} (m-3-2)
    (m-2-3) edge node [above] {$\phi_{B_{\Gamma^{+},\infty}}$} (m-2-4) edge node [left] {$\mu$} (m-3-3)
    (m-2-4) edge node [] {} (m-2-5) edge node [right] {$\mu_{B_{\Gamma^{+},\infty}}$} (m-3-4)
    (m-3-1) edge node [] {} (m-3-2)
    (m-3-2) edge node [above] {} (m-3-3)  edge node [left] {} (m-4-2)
    (m-3-3) edge node [above]
    {$\phi_{B_{\Gamma^{+}}}$} (m-3-4) edge node [left] {$\gamma$} (m-4-3)
    (m-3-4) edge node [] {} (m-3-5) edge node [right] {$\gamma_{B_{\Gamma^{+},\infty}}$} (m-4-4)
    (m-4-1) edge node [] {} (m-4-2)
    (m-4-2) edge node [above] {} (m-4-3) edge node [left] {} (m-5-2)
    (m-4-3) edge node [above] {$\phi_{Q}$} (m-4-4) edge node [left] {} (m-5-3)
    (m-4-4) edge node [] {} (m-4-5) edge node [left] {} (m-5-4);
\end{tikzpicture}
\end{equation}
\end{center}
We claim that exact sequences in this diagram and (\ref{diagram0}) are equivalent.
The middle exact sequences of (\ref{diagram0}) and (\ref{diagram1}) are trivially equivalent via the isomorphism
$\mathfrak{T}:B_{\Gamma^{+}}\times_{\tau}^{\iso}\Gamma^{+}\rightarrow \T(\Gamma)$.
By viewing $B_{\Gamma^{+}}$ as the algebra of functions that have limit, the map $f \in B_{\Gamma^{+}} \mapsto \lim_{x\in\Gamma^{+}} f(x)$
induces an isomorphism $B_{\Gamma^{+}}/B_{\Gamma^{+},\infty} \rightarrow \C$, which intertwines the action $\tilde{\tau}$ and the trivial action $\id$ on $\C$.
So $(B_{\Gamma^{+}}/B_{\Gamma^{+},\infty},\Gamma^{+},\tilde{\tau})\simeq (\C,\Gamma^{+},\id)$.
Moreover, $\mathfrak{T}$ combines with the isomorphism
$h: B_{\Gamma^{+}}/B_{\Gamma^{+},\infty}\times_{\tilde{\tau}}^{\iso}\Gamma^{+} \rightarrow \C\times_{\id}^{\iso}\Gamma^{+}
\rightarrow C^{*}(\Gamma)\simeq C(\hat{\Gamma})$
to identify the right-hand exact sequence equivalently to
$0\rightarrow {\mathcal C}_{\Gamma} \rightarrow \T(\Gamma) \stackrel{\psi_{T^{*}}}{\rightarrow} C(\hat{\Gamma}) \rightarrow 0$.

For the bottom sequence, we consider the pair of
\[ \iota_{\C}: z\in \C \mapsto z 1_{\T(\Gamma)}  \text{ and } \iota_{\Gamma^{+}}: x\in\Gamma^{+}\mapsto T_{x}^{*} \in\T(\Gamma). \]
It is a partial-isometric covariant representation, such that  $(\T(\Gamma),\iota_{\C},\iota_{\Gamma^{+}})$ is a partial-isometric crossed product of
$(\C,\Gamma^{+},\id)$.
So we have an isomorphism
\[ \Upsilon: Q^{\piso}\rightarrow  \C\times_{\id}^{\piso}\Gamma^{+}\stackrel{\iota}{\rightarrow}\T(\Gamma) \text{ in which }
\Upsilon(i_{\Gamma^{+}}(x))=T_{x}^{*} \text{ for all } x,\]
and moreover if $(j_{Q},u)$ denotes the canonical covariant partial-isometric representation of the system
$(Q:=B_{\Gamma^{+}}/B_{\Gamma^{+},\infty},\Gamma^{+},\tilde{\tau})$ in $Q^{\piso}$, then $\Upsilon$ satisfies the equations
$\Upsilon(u_{x})=T_{x}^{*}$ and $\Upsilon(j_{Q}(1_{x}+B_{\Gamma^{+},\infty}))= \iota_{\C}(\lim _{y} 1_{x}(y))=1$ for all $x\in\Gamma^{+}$.
To see $\Upsilon(\ker\phi_{Q})={\mathcal C}_{\Gamma}$, recall from \cite[Proposition 2.3]{AZ} that
\[ \ker\phi_{Q}:=\overline{\newspan}\{u_{x}^{*}j_{Q}(a)(1-u_{z}^{*}u_{z})u_{y} : a\in Q, x,y,z\in\Gamma^{+}\}. \]
Since $\Upsilon(u_{x}^{*}j_{Q}(a)(1-u_{z}^{*}u_{z})u_{y})$ is a scalar multiplication of $T_{x}(1-T_{z}T_{z}^{*})T_{y}^{*}$,
therefore $\Upsilon(\ker\phi_{Q})={\mathcal C}_{\Gamma}$.
Consequently the two exact sequences are equivalent:
\begin{equation*}
\begin{diagram}\dgARROWLENGTH=0.3\dgARROWLENGTH
\node{0} \arrow{e}\node{\ker\phi_{Q}} \arrow{s,l}{\Upsilon}\arrow{e}
\arrow{s}\arrow{e}\node{Q^{\piso}}\arrow{s,l}{\Upsilon}\arrow{e,t}{\phi_{Q}}
\node{Q^{\iso}}\arrow{s,l}{h}\arrow{e} \node{0}\\
\node{0} \arrow{e} \node{{\mathcal C}_{\Gamma}} \arrow{e} \node {\T(\Gamma)} \arrow{e,t}{\psi_{T}} \node {C(\hat{\Gamma})}
\arrow{e} \node{0.}
\end{diagram}
\end{equation*}

For the second column exact sequence, we note that the isomorphism $\jmath:Q^{\piso}\simeq\C\times^{\piso}_{\id}\Gamma^{+} \rightarrow \T(\Gamma)$
satisfies $\jmath\circ \gamma=\varphi_{T^{*}}$.
This implies
\[ B_{\Gamma^{+},\infty}\times_{\tau}^{\piso}\Gamma^{+}\simeq\ker(\jmath\circ\gamma) =\ker \varphi_{T^{*}}, \]
and therefore the second column sequence of diagram (\ref{diagram0}) is equivalent to
$0\rightarrow \ker\varphi_{T^{*}}\rightarrow B_{\Gamma^{+}}\times_{\tau}^{\piso}\Gamma^{+}\rightarrow \T(\Gamma) \rightarrow 0$.

Next we are working for the first row.
The homomorphism $\phi_{B_{\Gamma^{+}}}$ in the following diagram
\begin{equation*}
\begin{diagram}\dgARROWLENGTH=0.7\dgARROWLENGTH
\node{B_{\Gamma^{+}}\times_{\tau}^{\piso}\Gamma^{+}} \arrow{se,l}{\varphi_{T}}\arrow{e,t}{\phi_{B_{\Gamma^{+}}}}
\node{B_{\Gamma^{+}}\times_{\tau}^{\iso}\Gamma^{+}} \arrow{s,l}{\mathfrak{T}}\\
\node{} \node{\T(\Gamma),}
\end{diagram}
\end{equation*}
restricts to the homomorphism $\phi_{B_{\Gamma^{+},\infty}}$ of the ideal $B_{\Gamma^{+},\infty}\times_{\tau}^{\piso}\Gamma^{+}\simeq \ker\varphi_{T^{*}}$ onto
$B_{\Gamma^{+},\infty}\times_{\tau}^{\iso}\Gamma^{+}\simeq {\mathcal C}_{\Gamma}$.
So the homomorphism $\varphi_{T}|:\ker\varphi_{T^{*}}\rightarrow {\mathcal C}_{\Gamma}$
has kernel $I:=\ker\varphi_{T^{*}}\cap\ker\varphi_{T}$, and therefore first row exact sequence of the two diagrams are indeed equivalent.

Finally we show that such $\Psi$ exists.
Consider $C(\hat\Gamma)\simeq C^{*}(\Gamma)\simeq \C\times_{\id}\Gamma$ is the $C^*$-algebra generated by the unitary
representation $x\in\Gamma \mapsto \delta_{x} \in \C\times_{\id}\Gamma$.
Then we have a homomorphism $\pi_{\delta^{*}}\times\delta^{*}: B_{\Gamma^{+}}\times_{\tau}^{\piso}\Gamma^{+} \rightarrow \C\times_{\id}\Gamma$
which satisfies $\pi_{\delta^{*}}\times\delta^{*}(v_{x})=\delta_{x}^{*}$ for all $x\in\Gamma^{+}$, and hence it is surjective.
By looking at the spanning elements of $\ker\varphi_{T}$ and $\ker\varphi_{T^{*}}$ we can see that these two ideals are contained in
$\ker (\pi_{\delta^{*}}\times\delta^{*})$, therefore $\J:=\ker\varphi_{T}+\ker\varphi_{T^{*}}$ must be also in $\ker (\pi_{\delta^{*}}\times\delta^{*})$.
For the other inclusion, let $\rho$ be a unital representation of $B_{\Gamma^{+}}\times_{\tau}^{\piso}\Gamma^{+}$ on a Hilbert space $H_{\rho}$ with
$\ker \rho=\J$.
Then for $s\in\Gamma^{+}$ we have
$\rho ((1-v_{s}v_{s}^{*})-(1-v_{s}^{*}v_{s}))=0$ because $1-v_{s}v_{s}^{*}\in\ker\varphi_{T^{*}}$ and $1-v_{s}^{*}v_{s}\in \ker\varphi_{T}$ belong to $\J$.
So $0=\rho(v_{s}^{*}v_{s}-v_{s}v_{s}^{*})$, which implies that $\rho(v_{s}^{*}v_{s})=\rho(v_{s}v_{s}^{*})$.
On the other hand the equation
$\rho ((1-v_{s}v_{s}^{*})+(1-v_{s}^{*}v_{s}))=0$ gives $\rho(v_{s}v_{s}^{*})=I$.
Therefore $\rho(v_{s}v_{s}^{*})=\rho(v_{s}^{*}v_{s})=I$, and this means $\rho(v_{s})$ is unitary for every $s\in\Gamma^{+}$.
Consequently a representation $\tilde{\rho}:\C\times_{\id}\Gamma \rightarrow B(H_{\rho})$ exists, and it satisfies
$\tilde{\rho}\circ (\pi_{\delta^{*}}\times \delta^{*})=\rho$.
Thus $\ker \pi_{\delta^{*}} \times \delta^{*}\subset \ker \rho=\J$, and the composition $\pi_{\delta^{*}} \times \delta^{*}$ with
the Fourier transform $C^{*}(\Gamma)\simeq C(\hat{\Gamma})$ is the wanted homomorphism $\Psi$.
\end{proof}

\section{The Primitive Ideals of $\c\times_{\tau}^{\piso}\N$}

Suppose $\Gamma^{+}$ is now the additive semigroup $\N$.
The algebra $B_{\N}$ is conveniently viewed as the $C^*$-algebra $\c$ of convergent sequences, the ideal $B_{\N,\infty}$ with $\c_{0}$, and
the action $\tau$ of $\N$ on $\c$ is generated by the unilateral shift:
$\tau_{1}(x_{0},x_{1},x_{2},\cdots)=(0,x_{0},x_{1},x_{2},\cdots)$.
The universal $C^*$-algebra $\c\times_{\tau}^{\piso}\N$ is generated by a power partial isometry $v:=i_{\N}(1)$.
The Toeplitz algebra $\T(\Z)$ is the $C^*$-subalgebra of $B(\ell^{2}(\N))$ generated by isometries $\{T_{n}: n\in\N\}$,
where $T_{n}(e_{i})=e_{n+i}$ on the set of usual orthonormal basis $\{e_{i} :  i\in \N\cup\{0\} \}$ of $\ell^{2}(\N)$, and
the commutator ideal of $\T(\Z)$ is $\K(\ell^{2}(\N))$.
Kernels of $\varphi_{T}$ and $\varphi_{T^{*}}$ are identified in \cite[Lemma 6.2]{LR} by
\[ \ker \varphi_{T}=\overline{\newspan}\{g_{i,j}^{m}: i,j,m\in\N\}; \quad \ker \varphi_{T^{*}}=\overline{\newspan}\{f_{i,j}^{m}: i,j,m\in\N\} \]
where
\[ g_{i,j}^{m}=v_{i}^{*} v_{m}v_{m}^{*}(1-v^{*}v)v_{j} \quad \text{and} \quad f_{i,j}^{m}=v_{i} v_{m}^{*}v_{m}(1-vv^{*})v_{j}^{*}. \]
Moreover  $\I:=\ker \varphi_{T}\cap \ker \varphi_{T^{*}}$ is an essential ideal in $\c\times_{\tau}^{\piso}\N$ \cite[Lemma 6.8]{LR}, given by
\[ \overline{\newspan}\{f_{i,j}^{m}-f_{i,j}^{m+1}=g_{m-i,m-j}^{m}-g_{m-i,m-j}^{m+1}: m\in\N, 0\le i,j \le m\}. \]
The main point of \cite[\S 6]{LR} is to show that there exist isomorphisms of $\ker \varphi_{T}$ and $\ker \varphi_{T^{*}}$ onto the algebra
\[ \A:=\{f: \N\rightarrow K(\ell^{2}(\N)) : f(n) \in P_{n}K(\ell^{2}(\N))P_{n} \text{ and } \varepsilon_{\infty}(f)=\lim_{n}f(n) \text{ exists} \}, \]
where $P_{n}:=1-T_{n+1}T_{n+1}^{*}$ is the projection of $\ell^{2}(\N)$ onto the subspace spanned by $\{e_{i}: i=0,1,2,\cdots,n\}$, and such that
they restrict to isomorphisms of $\I$ onto the ideal
\[ \A_{0}:=\{f\in\A : \lim_{n} f(n) =0\} \text{ of } \A. \]

We shall show in Proposition \ref{A-A0} that $\A$ and $\A_{0}$ are related to the algebras of compact operators on the Hilbert $\c$-module $\ell^{2}(\N,\c)$  and on the closed sub-$\c$-module $\ell^{2}(\N,\c_{0})$.
We supply our readers with some basic theory of the $C^*$-algebra of operators on this Hilbert module, and let us begin with recalling the module structure of $\ell^{2}(\N,\c)$ (and its closed sub-module).
The vector space $\ell^{2}(\N,\c)$, containing all $\c$-valued functions $\textsf{a}:\N\rightarrow\c$ such that the series $\sum_{n\in\N}\textsf{a}(n)^{*}\textsf{a}(n)$ converges in the norm of $\c$, forms a Hilbert $\c$-module with the module structure defined by
$(\textsf{a}\cdot x)(n)=\textsf{a}(n)x$ for $x\in\c$,
and its $\c$-valued inner product given by $\langle \textsf{a},\textsf{b}\rangle=\sum_{n\in\N}\textsf{a}(n)^{*}\textsf{b}(n)$.
In fact the module $\ell^{2}(\N,\c)$ is naturally isomorphic to the Hilbert module $\ell^{2}(\N)\otimes\c$ that arises from the completion of algebraic (vector space) tensor product $\ell^{2}(\N)\odot\c$ associated to the $\c$-valued inner product defined on simple tensor product by
$\langle \xi\otimes x, \eta\otimes y\rangle=\langle \xi,\eta\rangle x^{*} y$ for $\xi,\eta \in \ell^{2}(\N)$ and $x, y \in\c$.
The isomorphism is implemented by the map $\phi$ that takes $(e_{i}\otimes x) \in \ell^{2}(\N)\otimes \c$ to the element $\phi(e_{i}\otimes x)\in \ell^{2}(\N,\c)$
which is the function $[\phi(e_{i}\otimes x)](n)=\left\{ \begin{array}{ll} x & \mbox{ if } i=n \\ 0 & \mbox{ otherwise. } \end{array} \right.$
By exactly the same arguments, we see that the two Hilbert $\c_{0}$-modules $\ell^{2}(\N,\c_{0})$ and $\ell^{2}(\N)\otimes\c_{0}$ are isomorphic.
However since $\c_{0}$ is an ideal of $\c$, it follows that the $\c_{0}$-module $\ell^{2}(\N,\c_{0})$ is a closed sub-$\c$-module of $\ell^{2}(\N,\c)$, and respectively  $\ell^{2}(\N)\otimes\c_{0}$ is a closed sub-$\c$-module of $\ell^{2}(\N)\otimes\c$.
Moreover the $\c$-module isomorphism $\phi$ restricts to $\c_{0}$-module isomorphism $\ell^{2}(\N,\c_{0})\simeq \ell^{2}(\N)\otimes\c_{0}$.

Next, we consider the $C^*$-algebra $\L(\ell^{2}(\N,\c))$ of adjointable operators on $\ell^{2}(\N,\c)$, and the ideal $\K(\ell^{2}(\N,\c))$ of $\L(\ell^{2}(\N,\c))$ spanned by the set $\{\theta_{\textsf{a},\textsf{b}}: \textsf{a},\textsf{b} \in \ell^{2}(\N,\c)\}$ of compact operators on the module
$\ell^{2}(\N,\c)$.
The algebra $\K(\ell^{2}(\N,\c_{0}))$ is defined by the same arguments, and note that $\K(\ell^{2}(\N,\c_{0}))$ is an ideal of $\K(\ell^{2}(\N,\c))$.
The isomorphism of two modules $\ell^{2}(\N,\c)$ and $\ell^{2}(\N)\otimes\c$, implies that $\K(\ell^{2}(\N,\c))\simeq  \K(\ell^{2}(\N)\otimes\c)$, which by the Hilbert module theorem, this is the $C^*$-algebraic tensor product $\K(\ell^{2}(\N))\otimes\c$ of $\K(\ell^{2}(\N))$ and $\c$.
We shall often use the characteristics functions $\{1_{n} : n\in\N\}$ as generator elements of $\c$ and
the spanning set $\{\theta_{e_{i}\otimes 1_{n},e_{j}\otimes 1_{n}} : i,j,n \in \N\}$ of $\K(\ell^{2}(\N,\c))$ in our computations.

There is another ingredient that we need to consider to state the Proposition.
Suppose $S\in \L(\ell^{2}(\N,\c))$ is an operator defined by $S(\textsf{a})(i)=\textsf{a}(i-1)$ for $i\ge 1$ and zero otherwise.
One can see that $S^{*}S=1$, i.e. $S$ is an isometry.
Let $p\in \L(\ell^{2}(\N,\c))$ be the projection $(p(\textsf{a}))(n)=1_{n}\textsf{a}(n)$ for $\textsf{a} \in \ell^{2}(\N,\c)$, and similarly
$q\in \L(\ell^{2}(\N,\c_{0}))$ be the projection $(q (\textsf{a}))(n)=1_{n}\textsf{a}(n)$ for $\textsf{a}\in \ell^{2}(\N,\c_{0})$.
Then the following two partial isometric representations of $\N$ in $p\L(\ell^{2}(\N,\c))p$ defined by
\[ w: n\in\N \mapsto pS_{n}^{*}p \quad \text{ and } \quad t: n\in\N \mapsto pS_{n}p, \]
induce the representations $\pi_{w}\times w$ and $\pi_{t}\times t$ of
$\c\times_{\tau}^{\piso}\N$ in $p\L(\ell^{2}(\N,\c))p$ which satisfy
$\pi_{w}\times w(v_{i})=pS_{i}^{*}p$ and $\pi_{t}\times t(v_{i})=pS_{i}p$ for all $i\in\N$.
These $\pi_{w}\times w$ and $\pi_{t}\times t$ are faithful representations \cite[Example 4.3]{AZ}.

\begin{prop}\label{A-A0}
The representations $\pi_{w}\times w$ and $\pi_{t}\times t$ map $\ker\varphi_{T}$ and $\ker\varphi_{T^{*}}$ isomorphically onto
the full corner $p\K(\ell^{2}(\N,\c))p$.
Moreover, they restrict to isomorphisms of the ideal $\ker\varphi_{T}\cap\ker\varphi_{T^{*}}$ onto the full corner $q\K(\ell^{2}(\N,\c_{0}))q$.
\end{prop}

\begin{remark}
It follows from this Proposition that $\Prim \ker \varphi_{T}$ and $\Prim \ker \varphi_{T^{*}}$ are both homeomorphic to $\Prim \c$.
In fact,  since $\ker\varphi_{T^{*}}\simeq \c_{0}\times_{\tau}^{\piso}\N$ by \cite[Corollary 3.1]{Adji-Abbas}, we can
therefore deduce that $\c_{0}\times_{\tau}^{\piso}\N$ is Morita equivalent to $\K(\ell^{2}(\N,\c))$.
This is a useful fact for our subsequential work on the partial-isometric crossed product of lattice semigroup $\N\times \N$.
\end{remark}

\begin{proof}[Proof of Proposition \ref{A-A0}]
We only have to show that $\pi_{t}\times t(\ker\varphi_{T^{*}})=p\K(\ell^{2}(\N,\c))p$ and
$\pi_{t}\times t(\ker\varphi_{T}\cap\ker\varphi_{T^{*}})=q\K(\ell^{2}(\N,\c_{0}))q$.
The rest of arguments is done in \cite[Example 4.3]{AZ}.

Note that the algebra  $p\K(\ell^{2}(\N,\c))p$ is spanned by $\{p\theta^{\c}_{e_{i}\otimes 1_{n},e_{j}\otimes 1_{n}}p : i,j,n\in\N\}$.
Since $\pi_{t}\times t(f_{i,j}^{n})=p\theta^{\c}_{e_{i}\otimes 1_{n},e_{j}\otimes 1_{n}}p$ for every $i,j,n\in\N$, therefore
$\pi_{t}\times t(\ker\varphi_{T^{*}})=p\K(\ell^{2}(\N,\c))p$.

Similarly we consider that $\{q\theta^{\c_{0}}_{e_{i}\otimes 1_{\{n\}},e_{j}\otimes 1_{\{n\}}}q : i,j \le n\in\N\}$
spans $q\K(\ell^{2}(\N,\c_{0}))q$.
We use the equation
$\theta^{\c}_{e_{i}\otimes 1_{n},e_{j}\otimes 1_{0}}=\theta^{\c}_{e_{i}\otimes 1_{n},e_{j}\otimes 1_{n}}$ for every $n\in\N$, in the computations below,
to see that
\begin{align*}
\pi_{t}\times t(f_{i,j}^{n}-f_{i,j}^{n+1}) & = p (\theta^{\c}_{e_{i}\otimes 1_{n},e_{j}\otimes 1_{0}} - \theta^{\c}_{e_{i}\otimes 1_{n+1},e_{j}\otimes 1_{0}})p \\
& = p (\theta^{\c}_{(e_{i}\otimes 1_{n})-(e_{i}\otimes 1_{n+1}),(e_{j}\otimes 1_{0})}) p = p  (\theta^{\c}_{e_{i}\otimes 1_{\{n\}},e_{j}\otimes 1_{0}}) p \\
& = p (\theta^{\c}_{e_{i}\otimes 1_{\{n\}},e_{j}\otimes 1_{\{n\}}}) p.
\end{align*}
To convince that every  $p (\theta^{\c}_{e_{i}\otimes 1_{\{n\}},e_{j}\otimes 1_{\{n\}}}) p$ belongs to $q\K(\ell^{2}(\N,\c_{0}))q$,
we need the embedding $\iota^{\K}$ of $q\K(\ell^{2}(\N,\c_{0}))q$ in $p\K(\ell^{2}(\N,\c))p$ stated in Lemma \ref{K}.
In fact, every element $p (\theta^{\c}_{e_{i}\otimes 1_{\{n\}},e_{j}\otimes 1_{\{n\}}}) p$ spans $\iota^{\K}(q\K(\ell^{2}(\N,\c_{0}))q)$,
therefore  $\pi_{t}\times t(\ker \varphi_{T^{*}}\cap\ker\varphi_{T})=\iota^{\K}(q\K(\ell^{2}(\N,\c_{0}))q)$.
\end{proof}

\begin{lemma}\label{K}
Let $p\in \L(\ell^{2}(\N,\c))$ and $q\in \L(\ell^{2}(\N,\c_{0}))$ are the projections defined by
$(p(\textsf{a}))(n)=1_{n}\textsf{a}(n)$ for $\textsf{a} \in \ell^{2}(\N,\c)$, and
$(q (\textsf{a}))(n)=1_{n}\textsf{a}(n)$ for $\textsf{a}\in \ell^{2}(\N,\c_{0})$.
Then the full corner $q\K(\ell^{2}(\N,\c_{0}))q$  embeds naturally via
$\iota^{\K}(q\theta^{\c_{0}}_{\textsf{a},\textsf{b}}q) = p \theta^{\c}_{\textsf{a},\textsf{b}}p$ as an ideal
in $p\K(\ell^{2}(\N,\c))p$, and there exists a short exact sequence
\[ 0 \longrightarrow q\K(\ell^{2}(\N,\c_{0}))q \stackrel{\iota^{\K}}{\longrightarrow} p\K(\ell^{2}(\N,\c))p
\stackrel{q^{\K}}{\longrightarrow} \K(\ell^{2}(\N))\longrightarrow 0, \]
where $q^{\K}(p\theta^{\c}_{\textsf{a},\textsf{b}}p)=\theta_{x,y}$ with $x,y \in \ell^{2}(\N)$ are given by
$x_{i}=\lim_{n\rightarrow\infty} (1_{i}\textsf{a}(i))(n)$ and
$y_{i}=\lim_{n\rightarrow\infty} (1_{i}\textsf{b}(i))(n)$.
In particular we have
\begin{align*}
q^{\K}(p\theta^{\c}_{e_{i}\otimes 1_{n}, e_{j}\otimes 1_{m}}p) = q^{\K}(\theta^{\c}_{p(e_{i}\otimes 1_{n}), p(e_{j}\otimes 1_{m})})
& = q^{\K}(\theta^{\c}_{e_{i}\otimes 1_{n\vee i}, e_{j}\otimes 1_{m\vee j}}) \\
& = T_{i}(1-TT^{*})T_{j}^{*} \in \K(\ell^{2}(\N)).
\end{align*}
\end{lemma}

\begin{proof}
Apply \cite[Lemma 2.6]{FMR} for the module $X:=\ell^{2}(\N,\c)$ and $I=\c_{0}$.
In this case we have the submodule $XI= \ell^{2}(\N,\c_{0})$.
Note that if $\textsf{a}\in \ell^{2}(\N,\c)$, then every sequence $\textsf{a}(i)\in\c$ is convergent in $\C$, and
the map $\textsf{q}:\textsf{a} \mapsto (\textsf{q}(\textsf{a}))(i)=\lim_{n\rightarrow\infty}(\textsf{a}(i))(n)$ gives
$0\rightarrow \ell^{2}(\N,\c_{0}) \rightarrow \ell^{2}(\N,\c) \stackrel{\textsf{q}}{\rightarrow} \ell^{2}(\N)\rightarrow 0$.
Moreover \cite[Lemma 2.6]{FMR}  proves that
$\iota^{\K}(\theta^{XI}_{\textsf{a},\textsf{b}})=\theta^{X}_{\textsf{a},\textsf{b}}$ and
$q^{\K}(\theta^{X}_{\textsf{a},\textsf{b}})=\theta^{X/XI}_{\textsf{q(\textsf{a})},\textsf{q(\textsf{b})}}$ give the exactness of the sequence
\[ 0\longrightarrow \K(\ell^{2}(\N,\c_{0})) \stackrel{\iota^{\K}}{\longrightarrow} \K(\ell^{2}(\N,\c)) \stackrel{q^{\K}}{\longrightarrow} \K(\ell^{2}(\N))
\longrightarrow 0. \]
Since $\iota^{\K}(q\theta^{XI}_{\textsf{a},\textsf{b}}q)=\theta^{X}_{q(\textsf{a}),q(\textsf{b})}=p\theta^{X}_{\textsf{a},\textsf{b}}p$ for every
$\textsf{a}$ and $\textsf{b}$ in $\ell^{2}(\N,\c_{0})$,
the corner  $q\K(\ell^{2}(\N,\c_{0}))q$ is embedded into $p\K(\ell^{2}(\N,\c))p$ such that $q^{\K}$ is defined by
$q^{\K}(p\theta^{X}_{\textsf{a},\textsf{b}}p)=q^{\K}(\theta^{X}_{p(\textsf{a}),p(\textsf{b})})=\theta_{x,y}$ where
$x_{i}=\lim_{n\rightarrow\infty}(1_{i}\textsf{a}(i))(n)$ and $y_{i}=\lim_{n\rightarrow\infty}(1_{i}\textsf{b}(i))(n)$.
Thus we obtain the required exact sequence.
\end{proof}

\begin{prop}\label{c0-corner}
There are isomorphisms $\Theta:p\K(\ell^{2}(\N,\c))p\rightarrow\ker \varphi_{T}$ and $\Theta_{*}:p\K(\ell^{2}(\N,\c))p\rightarrow\ker \varphi_{T^{*}}$ defined by $\Theta(p\theta^{\c}_{e_{i}\otimes 1_{n},e_{j}\otimes 1_{n}}p)=g_{i,j}^{n}$ and
$\Theta_{*}(p\theta^{\c}_{e_{i}\otimes 1_{n},e_{j}\otimes 1_{n}}p)=f_{i,j}^{n}$ for all $i,j,n \in\N$ such that the following commutative diagram has all rows and columns exact:
\begin{center}
\begin{equation}\label{diagram00}
\begin{tikzpicture}
  \matrix(m)[matrix of math nodes,row sep=2em,column sep=2em,minimum width=2em]
  {\empty & 0 & 0 & 0 & \empty\\
  0 & q\K(\ell^{2}(\N,\c_{0}))q & p\K(\ell^{2}(\N,\c))p & \K(\ell^{2}(\N)) & 0\\
   0 & p\K(\ell^{2}(\N,\c))p & \c\times_{\tau}^{\piso}\N & \mathcal{T}(\Z) & 0\\
   0 & \K(\ell^{2}(\N)) & \mathcal{T}(\Z) & C(\TT) & 0\\
   \empty & 0 & 0 & 0 & \empty\\};
  \path[-stealth]
    (m-1-2) edge node [above] {} (m-2-2)
    (m-1-3) edge node [above] {} (m-2-3)
    (m-1-4) edge node [above] {} (m-2-4)
    (m-2-1) edge node [above] {} (m-2-2)
    (m-2-2) edge node [above] {$\i^{\K}$} (m-2-3) edge node [left] {$\i^{\K}\circ\alpha$} (m-3-2)
    (m-2-3) edge node [above] {$q^{\K}$} (m-2-4) edge node [right] {$\Theta_{*}$} (m-3-3)
    (m-2-4) edge node [] {} (m-2-5) edge node [left] {} (m-3-4)
    (m-3-1) edge node [] {} (m-3-2)
    (m-3-2) edge node [above] {$\Theta$} (m-3-3)  edge node [left] {$q^{\K}$} (m-4-2)
    (m-3-3) edge node [above] {$\Psi$} (m-4-4) edge node [above] {$\varphi_{T}$} (m-3-4) edge node [left] {$\varphi_{T^*}$} (m-4-3)
    (m-3-4) edge node [] {} (m-3-5) edge node [left] {$\psi_{T^*}$} (m-4-4)
    (m-4-1) edge node [] {} (m-4-2)
    (m-4-2) edge node [above] {} (m-4-3) edge node [left] {} (m-5-2)
    (m-4-3) edge node [above] {$\psi_{T}$} (m-4-4) edge node [left] {} (m-5-3)
    (m-4-4) edge node [] {} (m-4-5) edge node [left] {} (m-5-4);
\end{tikzpicture}
\end{equation}
\end{center}
\end{prop}
\begin{proof}
We apply Proposition \ref{BGamma} to the system $(\c,\N,\tau)$.
Let $\{v_{i} : i\in\N\}$ denote the generators of $\c\times_{\tau}^{\piso}\N$, and $\{\delta_{i}:i\in\Z\}$ the generator of $C^{*}(\Z)$.
Then the homomorphism $\Psi:\c\times_{\tau}^{\piso}\N\rightarrow C(\TT)$ given by Proposition \ref{BGamma} satisfies
$\Psi(v_{i})=\delta^{*}_{i}=(z\mapsto \overline{z}^{i}) \in C(\TT)$ for every $i\in\N$.
Moreover $\Theta=(\pi_{w}\times w)^{-1}$ and $\Theta_{*}=(\pi_{t}\times t)^{-1}$, by Proposition \ref{A-A0},  satisfy
$\Theta(p\theta^{\c}_{e_{i}\otimes 1_{n},e_{j}\otimes 1_{n}}p)=g_{i,j}^{n}$ and
$\Theta_{*}(p\theta^{\c}_{e_{i}\otimes 1_{n},e_{j}\otimes 1_{n}}p)=f_{i,j}^{n}$ for all $i,j,n \in\N$.
So the first row sequence is exact, and which is equivalent to the one of (\ref{diagram0}) for $\Gamma^{+}=\N$ because
\begin{equation*}
\begin{diagram}\dgARROWLENGTH=0.5\dgARROWLENGTH
\node{0} \arrow{e}\node{\I} \arrow{s,l}{\pi_{t}\times t}\arrow{e,t}{\id}
\arrow{s}\arrow{e}\node{\ker\varphi_{T^{*}}}\arrow{s,l}{\pi_{t}\times t}\arrow{e,t}{\varphi_{T}|}
\node{\K(\ell^{2}(\N))}\arrow{s,l}{\id}\arrow{e} \node{0}\\
\node{0} \arrow{e} \node{q\K(\ell^{2}(\N,\c_{0}))q} \arrow{e,t}{\i^{\K}} \node {p\K(\ell^{2}(\N,\c))p} \arrow{e,t}{q^{\K}} \node {\K(\ell^{2}(\N))}
\arrow{e} \node{0.}
\end{diagram}
\end{equation*}

For the first column we use the automorphism $\alpha$ of $q\K(\ell^{2}(\N,\c_{0}))q$ defined on its spanning element by
$\alpha(q\theta^{\c_{0}}_{e_{i}\otimes 1_{\{n\}},e_{j}\otimes 1_{\{n\}}}q)=q\theta^{\c_{0}}_{e_{n-i}\otimes 1_{\{n\}},e_{n-j}\otimes 1_{\{n\}}}q$.
Then by inspections on the spanning elements of the algebras involved, we can see that the diagram (\ref{diagram00}) commutes.
\end{proof}

Thus we know from the diagram that the set $\Prim \c\times_{\tau}^{\piso}\N$ is given by the sets
$\Prim \K(\ell^{2}(\N,\c))$ and $\Prim\T(\Z)$.
Since
\[ \Prim \T(\Z)=\Prim \K(\ell^{2}(\N))\cup \Prim C(\TT)=\{0\}\cup \TT, \] and
$\Prim \K(\ell^{2}(\N,\c))$ is homeomorphic to
\[\Prim \c=\Prim \c_{0}\cup \Prim \C \simeq\N\cup\{\infty\}, \]
therefore $\Prim \c\times_{\tau}^{\piso}\N$ consists of a copy of $\{I_{n}\}$ of $\N$ embedded as an open subset,
a copy of $\{J_{z}\}$ of $\TT$ embedded as a closed subset.
We identify these ideals in Proposition \ref{prim-a0} and Lemma \ref{prim-CT}.

Note for now that
$\ker\varphi_{T}$ and $\ker\varphi_{T^{*}}$ are primitive ideals of $\c\times_{\tau}^{\piso}\N$:
the Toeplitz representation $T$ of $\T(\Z)$ on $\ell^{2}(\N)$ is irreducible by \cite[Theorem 3.13]{murphy}, and
$\varphi_{T}$ and $\varphi_{T^{*}}$ are surjective homomorphisms of $\c\times_{\tau}^{\piso}\N$ onto $\T(\Z)$, so
$T\circ \varphi_{T}$ and $T\circ \varphi_{T^{*}}$ are irreducible representations of $\c\times_{\tau}^{\piso}\N$ on $\ell^{2}(\N)$.
Moreover, irreducibility of the representation
$\id\circ q^{\K}:p\K(\ell^{2}(\N,\c))p \stackrel{q^{\K}}{\rightarrow} \K(\ell^{2}(\N))\stackrel{\id}{\rightarrow} B(\ell^{2}(\N))$
implies the kernel $\I= \ker \varphi_{T} \cap \ker \varphi_{T^{*}}\simeq q\K(\ell^{2}(\N,\c_{0}))q$ of $\id\circ q^{\K}$
is a primitive ideal of $p\K(\ell^{2}(\N,\c))p \simeq \ker\varphi_{T}$.
Similarly, $\I$ is a primitive ideal of $\ker\varphi_{T^{*}}\simeq p\K(\ell^{2}(\N,\c))p$.
Although $\I\not\in \Prim {\bf c} \times_{\tau}^{\piso}\N$, the ideal $\I$ is essential in $\c\times_{\tau}^{\piso}\N$ by \cite[Lemma 6.8]{LR},
so the space $\Prim \I\simeq \Prim \c_{0}$ is dense in $\Prim {\bf c} \times_{\tau}^{\piso}\N$.

Next consider that $\K(\ell^{2}(\N))=\overline{\newspan}\{e_{ij}:=T_{i}(1-TT^{*})T_{j}^{*} : i,j \in\N\}$, and recall that there is
a natural isomorphism $\Lambda$ of $\K(\ell^{2}(\N,\c))\simeq \K(\ell^{2}(\N))\otimes\c$ onto the algebra
\[ C(\N\cup\{\infty\},\K(\ell^{2}(\N))):=\{f:\N\rightarrow \K(\ell^{2}(\N)) : \lim_{n} f(n) \text{ exists in } \K(\ell^{2}(\N))\} \] given by
$\Lambda(e_{ij}\otimes 1_{k})(n)= 1_{k}(n)e_{ij}$ for $i,j,k,n \in \N$.
Then $\Lambda(p\K(\ell^{2}(\N,\c))p)\subset \A$ because
\begin{align*}
[\Lambda(p (e_{ij}\otimes 1_{m})p)](n) & = [\Lambda(e_{ij}\otimes 1_{m\vee i\vee j})](n)  \\
& = \left\{ \begin{array}{ll} e_{ij} & \mbox{ if } n\ge m\vee i\vee j \\ 0 & \mbox{ otherwise } \end{array}\right. \\
& = \pi_{n}(f_{i,j}^{m})=\pi_{n}^{*}(g_{i,j}^{m}).
\end{align*}
Since $\Lambda=\pi\circ \Theta_{*}=\pi^{*}\circ\Theta$, $\Lambda$ maps the corners $p\K(\ell^{2}(\N,\c))p$
and  $q\K(\ell^{2}(\N,\c_{0}))q$ isomorphically onto the algebra $\A$ and $\A_{0}$ respectively.
Construction of this isomorphism in \cite[\S 6]{LR} involves the representations $\pi_{n}$ and $\pi_{n}^{*}$, for each $n\in\N$,
of  $\c\times_{\tau}^{\piso}\N$ on $\ell^{2}(\N)$ that are associated to the partial-isometric representations
$k\mapsto P_{n}T_{k}P_{n}$ and $k\mapsto P_{n}T_{k}^{*}P_{n}$ respectively,
where $P_{n}:=1-T_{n+1}T_{n+1}^{*}$ is the projection onto $H_{n}:=\newspan\{e_{i} : i=0,1,2,\cdots,n\}$.
For every $a\in\ker \varphi_{T^{*}}$, the sequence $\{\pi_{n}(a)\}_{n\in\N}$ is convergent in $\K(\ell^{2}(\N))$, and
then the map $a\in\ker \varphi_{T^{*}}\mapsto \pi(a):=\{\pi_{n}(a)\}_{n\in\N}\in\A$ defines the isomorphism.

These observations suggest that an extension of $\pi$ should give a representation  of $\c\times_{\tau}^{\piso}\N$ in the algebra
$C_\textrm{b}(\N,B(\ell^{2}(\N)))$, and then primitive ideals are the kernels of evaluation maps.
But we can consider a smaller algebra which gives more information on the image of $\pi$. Note that the algebra $C(\N\cup\{\infty\},B(\ell^{2}(\N)))$ is too small to consider, because the sequence $(P_{n}T_{k}P_{n})_{n\in\N}$ as we see, does not converge to $T_{k}$ in the operator norm on $B(\ell^{2}(\N))$, but it converges strongly to $T_{k}$.
Therefore we consider the set $C_\textrm{b}(\N\cup\{\infty\},B(\ell^{2}(\N))_{*-\textrm{s}})$ of functions $\xi:\N\rightarrow B(\ell^{2}(\N))$ such that $\lim_{n}\xi_{n}$ exists in the *-strong topology on $B(\ell^{2}(\N))$, and which satisfies $\|\xi\|_{\infty}:=\sup_{n}\|\xi_{n}\|<\infty$.
By \cite[Lemma 2.56]{RW}, it is a $C^*$-algebra with the pointwise operation from $B(\ell^{2}(\N))$ and the norm $\|\cdot\|_{\infty}$.
Then let
\begin{align*}
\mathcal{B}:= \{ f:\N\rightarrow B(\ell^{2}(\N))  & ~:~ \sup_{n\in\N} \|f(n)\|_{B(\ell^{2}(\N))}<\infty,~ f(n)\in P_{n}B(\ell^{2}(\N))P_{n} \text{ and } \\
 & \lim_{n\rightarrow\infty} f(n) \text{ exists in the *-strong topology on } B(\ell^{2}(\N)) \}. \end{align*}
Note that $\mathcal{B}$ is a subalgebra of $C_\textrm{b}(\N\cup\{\infty\},B(\ell^{2}(\N))_{*-\textrm{s}})$ because $P_{n}B(\ell^{2}(\N))P_{n}\simeq B(H_{n})$
is closed in $B(\ell^{2}(\N))$ for every $n\in\N$, and $\mathcal{B}$ has an identity $1_{\mathcal B}=(P_{0},P_{1},P_{2},\cdots)$.

\begin{prop}
There are faithful representations $\pi$ and $\pi^{*}$ of $\c\times_{\tau}^{\piso}\N$ in the algebra $\mathcal{B}$, which
defined on each generator $v_{k}\in \c\times_{\tau}^{\piso}\N$ by
\[ \pi(v_{k})(n):=\pi_{n}(v_{k})= P_{n}T_{k}P_{n} \text{ and } \pi^{*}(v_{k})(n):=\pi_{n}^{*}(v_{k})=P_{n}T_{k}^{*}P_{n} \text{ for } n\in\N.\]
These representations $\pi$ and $\pi^{*}$ are the extension of isomorphisms $\pi:\ker\varphi_{T^{*}}\rightarrow \A$ and
$\pi^{*}:\ker\varphi_{T}\rightarrow \A$ of \cite[Theorem 6.1]{LR}.
\end{prop}
\begin{proof}
The map $\pi$ is induced by the partial-isometric representation $k \mapsto W_{k}$ where $W_{k}(n)=P_{n}T_{k}P_{n}$,  and similarly for $\pi^{*}$
by $k \mapsto S_{k}$ where $S_{k}(n)=P_{n}T_{k}^{*}P_{n}$ for $n\in\N$.
These are unital representations: $\pi(1)=\pi(v_{0})=(P_{0},P_{1},P_{2},\cdots)=\pi^{*}(1)$.

By \cite[Proposition 5.4]{LR}, the representation $\pi$ is faithful if and only if for any $r>0$ and $i<j$ in $\N$, we have
$\xi_{i,j}^{r}\in \mathcal B$ for which
\[ \xi_{i,j}^{r}:=(\pi(1)-\pi(v_{r})^{*}\pi(v_{r}))(\pi(v_{i})\pi(v_{i})^{*}-\pi(v_{j})\pi(v_{j})^{*}) \text{ is a nonzero element of }  \mathcal B. \]
Let $r>0$  and $i<j \in\N$, then we consider the three cases $0<r \le i <j$, $i<r < j$ and $i<j\le r$ separately.
If $0<r \le i <j$, then
\begin{align*}
\xi_{i,j}^{r}(i) & =(P_{i}-\pi_{i}(v_{r})^{*}\pi_{i}(v_{r}))(\pi_{i}(v_{i})\pi_{i}(v_{i})^{*}-\pi_{i}(v_{j})\pi_{i}(v_{j})^{*}) \\
 & = (P_{i}-P_{i}T_{r}^{*}P_{i}T_{r}P_{i})(P_{i}T_{i}P_{i}T_{i}^{*}P_{i}-P_{i}T_{j}P_{i}T_{j}^{*}P_{i}) \\
 & = (P_{i}-P_{i}T_{r}^{*}T_{r}P_{i-r}P_{i})(P_{i}T_{i}T_{i}^{*}P_{i}-0) \\
 & = (P_{i}-P_{i-r})(P_{i}T_{i}T_{i}^{*}P_{i})
\end{align*}
and that $[\xi_{i,j}^{r}(i)](e_{i})=(P_{i}-P_{i-r})(e_{i})=e_{i}$.
If $i<j\le r$ then similar computations show that  $[\xi_{i,j}^{r}(i)](e_{i})=[P_{i}(P_{i}T_{i}T_{i}^{*}P_{i})](e_{i})=e_{i}$, and
for $i<r<j$ we have $[\xi_{i,j}^{r}(r)](e_{r})=(P_{r}-P_{0})(e_{r})=e_{r}$.
Thus $\xi_{i,j}^{r} \neq 0$ in $\mathcal B$.
The same outline of arguments is valid to show the representation $\pi^{*}$ is also faithful.
\end{proof}

So we have for every $n\in\N$ the representations $\pi_{n}=\varepsilon_{n}\circ \pi$ and $\pi_{n}^{*}=\varepsilon_{n}\circ \pi^{*}$
of $\c\times_{\tau}^{\piso}\N$ on $H_{n}$,
where $\varepsilon_{n}$ are the evaluation map of  $C_{\textrm{b}}(\N\cup\{\infty\}, B(\ell^{2}(\N))_{*-\textrm{s}})$.
Hence they are irreducible, indeed every nonzero vector of the subspace $H_{n}$ of $\ell^{2}(\N)$ is cyclic for $\pi_{n}^{*}$:
if $(h_{0},h_{1},\cdots,h_{n})\in H_{n}$ with $h_{j}\neq 0$ for some $j$,
then for every $i\in\{0,1,2,\cdots,n\}$, we have
\begin{align*}
(\pi_{n}^{*}(g_{i,j}^{n}))(h_{0},h_{1},\cdots,h_{n})  & = [T_{i}(1-TT^{*})T_{j}^{*}](h_{0},h_{1},\cdots,h_{n})\\
& = (0,\cdots,h_{j},\cdots,0),  \mbox{ where $h_{j}$ is in the $i$-th slot,}
\end{align*}
so $\pi_{n}^{*}(\frac{1}{h_{j}}~g_{i,j}^{n})(h)=e_{i}$, and therefore $H_{n}=\newspan\{\pi^{*}_{n}(\xi)h : \xi \in \c\times_{\tau}^{\piso}\N \}$.
Same arguments work for $\pi_{n}$.

Note for every $n\in\N$ that $\pi_{n}(f_{i,j}^{m})=e_{ij}=\pi_{n}(g_{n-i,n-j}^{k})$ for all $0\le i,j,m,k\le n$, and similarly
$\pi_{n}^{*}(g_{i,j}^{m})=e_{ij}=\pi_{n}^{*}(f_{n-i,n-j}^{k})$ for all $0\le i,j,m,k\le n$.
Thus every $f_{i,j}^{m}-g_{n-i,n-j}^{k}$ is contained in $\ker \pi_{n}$, and similarly $(g_{i,j}^{m}-f_{n-i,n-j}^{k})\in \ker \pi_{n}^{*}$.
We shall see many more elements of $\ker \pi_{n}$ as well as $\ker \pi_{n}^{*}$ in Proposition \ref{prim-a0}.

But now we recall that for $n\in\N$ the partial-isometric representation $J^{n}:\N\rightarrow B(H_{n})$ in \cite[\S 3]{LR}
defined by $J_{t}^{n}(e_{r})=\left\{ \begin{array}{ll} e_{t+r} & \text{ if } r+t \in \{0,1,\cdots,n\} \\ 0 & \text{ otherwise, }
\end{array}\right.$ induces the representation $\pi_{J^{n}}^{\N}\times J^{n}$ of $(\c\times_{\tau}^{\piso}\N,v)$ on $H_{n}$.
In fact $\pi_{J^{n}}^{\N}\times J^{n}=\pi_{n}$, because for every $k\in\N$ we have
$(\pi_{J^{n}}^{\N}\times J^{n}(v_{k}))(e_{r})=J^{n}_{k}(e_{r})=P_{n}T_{k}P_{n}(e_{r})$ where $r\in \{0,1,2,\cdots,n\}$.

The ideal $\ker\oplus_{r=0}^{n}\pi_{J^{r}}^{\N}\times J^{r}$ appears in the structure of $\c\times_{\tau}^{\piso}\N$
\cite[Lemma 5.7]{LR}. To be more precise about it, we need some results in \cite[\S 5]{LR} related to the system $(\C^{n+1},\tau,\N)$.
The crossed product $\C^{n+1}\times_{\tau}^{\piso}\N$ is the universal $C^*$-algebra generated by a canonical partial-isometric representation
$w$ of $\N$ such that $w_{r}=0$ for $r\ge n+1$.
Let $q_{n}: (\c\times_{\tau}^{\piso}\N,v)\rightarrow (\C^{n+1}\times_{\tau}^{\piso}\N,w)$ be the homomorphism induced by
$w:\N\rightarrow \C^{n+1}\times_{\tau}^{\piso}\N$, and note that it is surjective.
Then Lemma 5.7 of \cite{LR} shows that
$\ker q_{n}=\ker (\oplus_{r=0}^{n}\pi_{J^{r}}^{\N}\times J^{r})=\bigcap_{r=0}^{n} \ker(\pi_{J^{r}}^{\N}\times J^{r})$.
So by these arguments we obtain the following equation
\begin{equation}\label{qn}
 \ker q_{n}=\bigcap_{r=0}^{n}\ker\pi_{r} \text{ for every } n\in\N. \end{equation}

\begin{lemma}\label{Lnn}
For $n\in\N$, let $L_{n}$ be the ideal of $(\c\times_{\tau}^{\piso}\N,v)$ generated by $\{v_{r}: r\ge n+1\}$.
Then $L_{n}=\ker q_{n}$, and it is isomorphic to
\begin{equation}\label{LN}
 \{\xi \in \pi(\c\times_{\tau}^{\piso}\N)\subset C_{\textrm{b}}(\N\cup\{\infty\}, B(\ell^{2}(\N))_{*-\textrm{s}}) : \xi\equiv 0 \text{ on } \{0,1,2,\cdots,n\}\}. \end{equation}
\end{lemma}
\begin{proof}
We have $L_{n}\subset\ker q_{n}$ because $q_{n}(v_{k})=0$ for all $k\ge n+1$.
To see $\ker q_{n}\subset L_{n}$, let $\rho$ be a representation  of $\c\times_{\tau}^{\piso}\N$ on $H_{\rho}$ where $\ker \rho=L_{n}$.
Since $\rho(v_{t})=0$ for every $t\ge n+1$, by the universal property of $\C^{n+1}\times_{\tau}^{\piso}\N$, there exists a representation
$\tilde{\rho}$ of $\C^{n+1}\times_{\tau}^{\piso}\N$ on $H_{\rho}$ which satisfies $\tilde{\rho}\circ q_{n}=\rho$.
Thus $\ker q_{n}\subset \ker\rho= L_{n}$.

Next we show that $\pi(L_{n})$ and (\ref{LN}) are equal.
Let $r\ge n+1$, and consider $\pi(v_{r})$ is the sequence $(P_{i}T_{r}P_{i})_{i\in\N}$.
If $0\le i\le n$ then $0\le i+1\le n+1\le r$ and
\[ P_{i}T_{r}P_{i}=(1-T_{i+1}T_{i+1}^{*})T_{r}P_{i}=(1-T_{i+1}T_{i+1}^{*})T_{i+1}T_{r-(i+1)}P_{i}=0. \]
So $\pi(L_{n})$ is a subset of (\ref{LN}).
For the other inclusion, suppose $f\in \pi(\c\times_{\tau}^{\piso}\N)$ in which $f(i)=0$ for all $0\le i\le n$.
Since $f=\pi(\xi)$ for some $\xi \in \c\times_{\tau}^{\piso}\N$, and $\pi(\xi)(i)=\pi_{i}(\xi)=f(i)$ for all $i\in\N$,
we therefore have $\pi_{i}(\xi)= f(i)=0$ for all $0\le i\le n$.
Thus $\xi \in \cap_{i=0}^{n}\ker\pi_{i}=\ker q_{n}$, and hence $f=\pi(\xi)\in \pi(L_{n})$.
\end{proof}

Let $\pi_{\infty}:=\lim_{n}\pi_{n}$ and $\pi_{\infty}^{*}:=\lim_{n}\pi_{n}^{*}$ where the limits are taken with respect to the strong topology of
$B(\ell^{2}(\N))$.
Then $\pi_{\infty}$ and $\pi_{\infty}^{*}$ are the irreducible representations $\varphi_{T}:v_{k}\mapsto T_{k}$ and
$\varphi_{T^{*}}:v_{k}\mapsto T_{k}^{*}$ of $\c\times_{\tau}^{\piso}\N$ on $H_{\infty}:=\ell^{2}(\N)$.
Thus by \cite[Lemma 6.2]{LR} we have
\begin{align*}
\ker\pi_{\infty} & =\ker\varphi_{T}=\overline{\newspan}\{g_{i,j}^{m}:= v_{i}^{*} v_{m}v_{m}^{*}(1-v^{*}v)v_{j} : i,j,m\in\N\} \\
\ker\pi_{\infty}^{*} & =\ker\varphi_{T^{*}}=\overline{\newspan}\{f_{i,j}^{m}:=v_{i} v_{m}^{*}v_{m}(1-vv^{*})v_{j}^{*} : i,j,m\in\N\}.
\end{align*}

For $n\in\N$, let $\pi_{n}$ and $\pi_{n}^{*}$ be the irreducible representations of $\c\times_{\tau}^{\piso}\N$ on the subspace $H_{n}$ of $\ell^{2}(\N)$,
that are induced by the partial-isometric representations
$k\mapsto P_{n}T_{k}P_{n}$ and $k\mapsto P_{n}T_{k}^{*}P_{n}$.
Let $L_{n}$ be the ideal of $(\c\times_{\tau}^{\piso}\N,v)$ generated by $\{v_{r}: r\ge n+1\}$.
Then $\pi_{n}$ is the representation
\[ \varepsilon_{n}\circ \pi: \c\times_{\tau}^{\piso}\N\stackrel{\pi}{\longrightarrow} \B \subset C_\textrm{b}(\N\cup\{\infty\},B(\ell^{2}(\N))_{*-\textrm{s}})
\stackrel{\varepsilon_{n}}{\longrightarrow} B(H_{n}), \]
and similarly $\pi_{n}^{*}=\varepsilon_{n}\circ \pi^{*}$.
So $\ker\pi_{n}\simeq\ker\varepsilon_{n}\simeq\ker\pi_{n}^{*}$.

\begin{prop}\label{prim-a0}
Let $n\in\N$, then
\begin{itemize}
\item[(a)] $\ker\pi_{n}=\ker\pi_{n}^{*}\simeq\ker\varepsilon_{n}=\{\xi \in \B : \xi(n)=0\}$;
\item[(b)] $\ker\pi_{\infty}\simeq \ker\pi_{\infty}^{*}= \{\xi\in \B : ~  ^{*}\!\!-{\rm strong}\lim_{n} \xi(n)=0\}$.
\end{itemize}
\end{prop}

Furthermore
\begin{align*}
(c) & \ker \pi_{n}^{*}=\overline{\newspan}\{g_{i,j}^{m}-f_{n-i,n-j}^{k}+\eta : 0\le i,j,m,k\le n, \eta \in L_{n} \}, \\
    & \ker \pi_{n}=\overline{\newspan}\{f_{i,j}^{m}-g_{n-i,n-j}^{k}+\eta : 0\le i,j,m,k\le n, \eta \in L_{n} \}, \text{ and } \\
    & \ker \pi_{n}^{*}=\ker \pi_{n} \text{ for } n\in\N, \text{ in particular we have  } \ker \pi_{0}=\ker\pi_{0}^{*}=L_{0}; \\
(d) & \ker\pi_{n}|_{\ker\varphi_{T^{*}}}=\overline{\newspan}\{(f_{i,j}^{m}-f_{i,j}^{k})+ f_{x,y}^{z}: 0\le i,j,m,k\le n, \mbox{ one of } x,y,z \ge n+1 \}, \\
    & \ker\pi_{n}^{*}|_{\ker\varphi_{T}}=\overline{\newspan}\{(g_{i,j}^{m}-g_{i,j}^{k})+ g_{x,y}^{z}: 0\le i,j,m,k\le n, \mbox{ one of } x,y,z \ge n+1 \}, \\
    & \Theta_{*}^{-1}(\ker\pi_{n}|_{\ker\varphi_{T^{*}}})=\Theta^{-1}(\ker\pi_{n}^{*}|_{\ker\varphi_{T}}), \text{ and } \\
    & \ker\pi_{n}^{*}|_{\ker\varphi_{T}}\simeq \{ \textsf{a}\in \A : \textsf{a}(n)=0\}\simeq \ker\pi_{n}|_{\ker\varphi_{T^{*}}}; \\
(e) & \ker \pi_{n}^{*}|_{\I} = \overline{\newspan}\{g_{i,j}^{m}-g_{i,j}^{m+1}: 0\le i,j\le m \mbox{ \rm in } \N, \mbox{ \rm and } m\neq n\}
       = \ker \pi_{n}|_{\I} \\
    & =\overline{\newspan}\{f_{i,j}^{m}-f_{i,j}^{m+1}: 0\le i,j\le m \mbox{ \rm in } \N, \mbox{ \rm and } m\neq n\}
    \text{ is isomorphic to the ideal } \\
    & \{ \textsf{a}\in \A_{0} : \textsf{a}(n)=0\}.
\end{align*}

\begin{remark}\label{relateLR}
Note that the representations $\pi_{n}|_{\ker\varphi_{T^{*}}}$ and $\pi_{n}^{*}|_{\ker\varphi_{T}}$ are equivalent to the evaluation map
$\varepsilon_{n}:f\in \A \mapsto  f(n)\in B(H_{n})$ of $\A$ on $H_{n}$,
so we have $\ker\pi_{n}|_{\ker\varphi_{T^{*}}}\simeq \ker\pi_{n}^{*}|_{\ker\varphi_{T}}$ is isomorphic to $\{f\in\A : f(n)=0\}$, and
$\ker \pi_{n}|_{\I} =\ker \pi_{n}^{*}|_{\I}\simeq\{f\in\A_{0} : f(n)=0\}$; and
$\ker\pi_{\infty}\simeq \ker\pi_{\infty}^{*}\simeq \A$.
\end{remark}

\begin{proof}[Proof of Proposition \ref{prim-a0}]
Fix $n\in\N$.
We show for $\ker\pi_{n}$, and skip the proof for $\ker \pi_{n}^{*}$ because it contains the same arguments.
We clarify firstly that the space
\[ \J:=\overline{\newspan}\{f_{i,j}^{m}-g_{n-i,n-j}^{k}+\eta : 0\le i,j,m,k\le n, \eta \in L_{n} \} \]
is an ideal of $(\c\times_{\tau}\N,v)$  by showing $v\J\subset \J$ and $v^{*}\J\subset \J$.
Let $i=n$, then
\begin{align*}
v v_{k}v_{k}^{*}(1-v^{*}v)v_{n-j} =v (v^{*}v v_{k}v_{k}^{*})(1-v^{*}v)v_{n-j}
& =vv_{k}v_{k}^{*}v^{*}v(1-v^{*}v)v_{n-j} \\
& = v_{k+1}v_{k+1}^{*}(v-vv^{*}v)v_{n-j}=0,
\end{align*}
therefore
$v (f_{n,j}^{m}-g_{0,n-j}^{k}+\eta)  = v v_{n}v_{m}^{*}v_{m}(1-vv^{*})v_{j}^{*}- v v_{k}v_{k}^{*}(1-v^{*}v)v_{n-j}+v\eta
= f_{n+1,j}^{m}+v\eta$ belongs to $\J$ because $f_{n+1,j}^{m}\in L_{n}$.
If $0\le i\le n-1$, then $1\le i+1\le n$ and $n-i\ge 1$, and we have
\begin{align*}
vv_{n-i}^{*}v_{k}v_{k}^{*} & =vv^{*}v_{n-i-1}^{*}v_{n-i-1}v_{n-i-1}^{*}v_{k}v_{k}^{*}\\
& =v_{n-i-1}^{*}v_{n-i-1}vv^{*}v_{n-i-1}^{*}v_{k}v_{k}^{*} \\
& = v_{n-i-1}^{*}v_{n-i}v_{n-i}^{*}v_{k}v_{k}^{*}=v_{n-i-1}^{*}v_{\max\{n-i,k\}}v_{\max\{n-i,k\}}^{*},
\end{align*}
so
$v (f_{i,j}^{m}-g_{n-i,n-j}^{k}+\eta )= f_{i+1,j}^{m}-g_{n-(i+1),n-j}^{\max\{n-i,k\}}+v\eta \in \J$.

Now we check for $v^{*}\J$, and assume $i=0$, then
\[
v^{*}[f_{0,j}^{m}-g_{n,n-j}^{k}+\eta]  =v^{*}[v_{m}^{*}v_{m}(1-vv^{*})v_{j}^{*}-v_{n}^{*}v_{k}v_{k}^{*}(1-v^{*}v)v_{n-j}+\eta]
= 0 - g_{n+1,n-j}^{k}+v^{*}\eta \in\J
\]
because $g_{n+1,n-j}^{k}\in L_{n}$.
It follows by similar computations for $1\le i\le n$ that
\[ v^{*}[f_{i,j}^{m}-g_{n-i,n-j}^{k}+\eta] = f_{i-1,j}^{\max\{i,m\}}-g_{n-(i-1),n-j}^{k}+v^{*}\eta \in \J.\]

Next we show that $\J=\ker\pi_{n}$, one inclusion $\J\subset \ker\pi_{n}$ is clear  because
$\pi_{n}(f_{i,j}^{m})=\pi_{n}(g_{n-i,n-j}^{k})=T_{i}(1-TT^{*})T_{j}^{*}$ and
$L_{n}\subset\ker\pi_{n}$.
For the other inclusion, let $\sigma:\c\times_{\tau}^{\piso}\N \rightarrow B(H_{\sigma})$ be a nondegenerate representation with $\ker \sigma=\J$.
Note that $B(H_{n})=\newspan\{e_{ij}:=T_{i}(1-TT^{*})T_{j}^{*} : 0\le i,j\le n\}$.
Since $\{f_{i,j}^{n} : 0\le i,j \le n\}$ is a matrix-units for $B(H_{\sigma})$, there is a homomorphism $\psi$ of $B(H_{n})$ into $B(H_{\sigma})$ which
satisfies $e_{ij}\mapsto \sigma(f_{i,j}^{n})$.
Therefore $\sigma=\psi\circ\pi_{n}$, and hence  $\ker\pi_{n}\subset\ker\sigma=\J$.

Using the spanning elements of $\ker \pi_{n}$ and $\ker \pi_{n}^{*}$, and the equation
$f_{i,j}^{m}-g_{n-i,n-j}^{k}= - (g_{n-i,n-j}^{k}-f_{n-(n-i),n-(n-j)}^{m})$, we see that they contain each other, therefore
$\ker \pi_{n}=\ker \pi_{n}^{*}$ for every $n\in\N$.
The ideal $L_{0}$ is $\ker \pi_{0}=\ker\pi_{0}^{*}$ because $f_{0,0}^{0}-g_{0,0}^{0}= v^{*}v-vv^{*} \in L_{0}$.

For (d), let now $\J$ be $\overline{\newspan}\{(f_{i,j}^{m}-f_{i,j}^{k})+ f_{x,y}^{z}: 0\le i,j,m,k\le n, \mbox{ one of } x,y,z \ge n+1 \}$.
Then the same idea of calculations shows that $\J$ is an ideal of $\ker \varphi_{T^{*}}$, and
it is contained in $\ker\pi_{n}|_{\ker\varphi_{T^{*}}}$, then
for the other inclusion let $\sigma$ be a nondegenerate representation of $\ker \varphi_{T^{*}}$ such that $\ker \sigma=\J$,
get the homomorphism $\psi: B(H_{n})\rightarrow B(H_{\sigma})$ defined by $\psi(e_{ij})=\sigma(f_{i,j}^{n})$, and hence the equation
$\psi\circ\pi_{n}=\sigma$ implies that $\ker\pi_{n}|_{\ker\varphi_{T^{*}}}=\J$.
By computations on the spanning elements we see that the equation
$\Theta_{*}^{-1}(\ker\pi_{n}|_{\ker\varphi_{T^{*}}})=\Theta^{-1}(\ker\pi_{n}^{*}|_{\ker\varphi_{T}})$ is hold.
The same arguments work for the proof of (e), and we skip this.
\end{proof}

\begin{remark}
The map $n\in \N\cup \{\infty\} \mapsto I_{n}:=\ker \pi_{n}^{*}\in \Prim (\c\times_{\tau}^{\piso}\N)$ parameterizes the open subset
$\{P\in \Prim (\c\times_{\tau}^{\piso}\N) : \ker\varphi_{T}\simeq \A \not\subset P\}$ of $\Prim (\c\times_{\tau}^{\piso}\N)$ homeomorphic to
$\Prim \A$.
Note that the $\infty$ corresponds to the ideal $\ker\pi_{\infty}^{*}=\ker \varphi_{T^{*}}\in \Prim (\c\times_{\tau}^{\piso}\N)$, and it corresponds to
$\I=\ker \varphi_{T^{*}}|_{\ker\varphi_{T}}\in \Prim \A$.
\end{remark}

\begin{lemma}\label{property-In}
\begin{itemize}
\item [(i)] $\bigcap_{n=0}^{m} I_{n}=L_{m}$ for every $m\in\N$;
\item[(ii)] $\bigcap_{n\in\N} I_{n} = \{0\}$;
\item[(iii)] $\{0\}\varsubsetneq (\bigcap_{n>m}I_{n}) \subset \ker\pi_{\infty}^{*}\cap \ker\pi_{\infty}$ for every $m\in\N$.
\end{itemize}
\end{lemma}

\begin{proof}
Part (i) follows from (\ref{qn}) and Lemma \ref{Lnn}.
For (ii), note that $q_{\infty}$ is the identity map on $\c\times_{\tau}^{\piso} \N$, and that
$\oplus_{i\in\N}\pi_{i}=(\oplus_{i\in\N}(\pi_{J^{i}}^{\N}\times J^{i}))\circ ~\id$.
So $\bigcap_{n\in\N} I_{n} = \{0\}$ by faithfulness of $\oplus_{i\in\N}(\pi_{J^{i}}^{\N}\times J^{i})$ \cite[Corollary 5.5]{LR}.

The inclusion $\bigcap_{n>m}I_{n}\subset \ker\pi_{\infty}^{*}$ for every $m\in\N$ follows from the next arguments:
\begin{align*}
\bigcap_{n>m}\ker(\pi_{n}^{*}|_{\ker \pi_{\infty}}) & \simeq
\{f\in \A : f(n)=0 ~ \forall ~ n>m\} \\
& \subset \{f\in A : \lim_{n\rightarrow \infty} f(n)=0\} \\
& = \A_{0} \simeq \ker\pi_{\infty}^{*}|_{\ker \pi_{\infty}} \subset \ker\pi_{\infty}^{*} \in \Prim \c\times_{\tau}^{\piso}\N,
\end{align*}
so the two ideals $J:=\bigcap_{n>m} I_{n}$ and $L:=\ker\pi_{\infty}$ of $\c\times_{\tau}^{\piso}\N$ satisfy $J\cap L\subset \ker\pi_{\infty}^{*}$,
therefore either $J\subset \ker\pi_{\infty}^{*}$ or $L \subset \ker\pi_{\infty}^{*}$, but the latter is not possible. To show $J\subset \ker\pi_{\infty}$, since $\ker \pi_{n}=\ker \pi_{n}^{*}$ for each $n$, we act similarly using the fact that
\begin{align*}
\bigcap_{n>m}\ker(\pi_{n}|_{\ker \pi_{\infty}^{*}}) & \simeq
\{f\in \A : f(n)=0 ~ \forall ~ n>m\} \subset \ker\pi_{\infty} \in \Prim \c\times_{\tau}^{\piso}\N.
\end{align*}

Therefore, $J \subset \ker\pi_{\infty}^{*}\cap \ker\pi_{\infty}$. Moreover, since $g_{0,0}^{0}-g_{0,0}^{1}\neq 0$ which satisfies $\pi_{n}^{*}(g_{0,0}^{0}-g_{0,0}^{1})=0$ for all $n\geq 1$, it follows that $\{0\}\varsubsetneq (\bigcap_{n>m}I_{n})$.
\end{proof}

\begin{remark}
Part (ii) of Lemma \ref{property-In} confirms with the fact that $\I$ is an essential ideal of $\c\times_{\tau}^{\piso}\N$ \cite[Lemma 6.8]{LR}.
\end{remark}

Next consider for $z\in \TT$, the character $\gamma_{z}\in \hat{\Z}\simeq \TT$ defined by $\gamma_{z}: m\mapsto \overline{z}^{m}$.
Note that the map $\gamma_{z}:k\in \N \mapsto \gamma_{z}(k)$ is a partial-isometric representation of $\N$ in $\C\simeq B(\C)$.
Consequently for each $z\in \TT$, we have a representation $\pi_{\gamma_{z}}\times\gamma_{z}$ of  $\c\times_{\tau}^{\piso}\N$ on $\TT$ such that
$\pi_{\gamma_{z}}\times\gamma_{z}(v_{k})=\gamma_{z}(k)=\overline{z}^{k}$ for $k\in\N$, and it is irreducible.
Moreover we know that the homomorphism $\Psi:\c\times_{\tau}^{\piso}\N\rightarrow C(\TT)$
is the composition of the Fourier transform
$\C\times_{\id}\Z\simeq C^{*}(\Z)\simeq C(\TT)$ with $\ell\times\delta^{*}: \c\times_{\tau}^{\piso}\N\rightarrow \C\times_{\id}\Z$, in which
$\ell:(x_{n})\in\c\mapsto \lim_{n} x_{n}\in\C$ and $\delta$ is the unitary representation of $\Z$ on $\C\times_{\id}\Z$.

\begin{lemma}\label{prim-CT}
For $z\in \TT$, the character $\gamma_{z}: k\mapsto \overline{z}^{k}$ in $\hat{\Z}\simeq \TT$ gives an irreducible representation
$\pi_{\gamma_{z}}\times \gamma_{z}$ of $\c\times_{\tau}^{\piso}\N$ on $\TT$ such that
$\pi_{\gamma_{z}}\times \gamma_{z}=\varepsilon_{z}\circ (\ell\times\delta^{*})$.
Denote by $J_{z}$ the primitive ideal $\ker \pi_{\gamma_{z}}\times \gamma_{z}$ of $\c\times_{\tau}^{\piso}\N$.
Then $\ker\pi_{\infty}$ and $\ker\pi_{\infty}^{*}$ are contained in $J_{z}$ for every $z\in \TT$.
Moreover every ideal $I_{n}$ for $n\in\N$ is not contained in any $J_{z}$.
\end{lemma}
\begin{proof}
By using the Fourier transform we can view $\C\times_{\id}\Z\simeq C^{*}(\Z)$ as $C(\TT)$, and it follows that
$v_{k}\in \c\times_{\tau}^{\piso}\N$ is mapped into the function $\iota_{k}: t\mapsto \overline{t}^{k} \in C(\TT)$.

We know that primitive ideals of $C(\TT)$ are given by the kernels of evaluation maps $\varepsilon_{t}(f)=f(t)$ for $t\in\TT$, and
the character $\gamma_{z}$ is a partial-isometric representation of $\N$ on $\TT$ for $z\in\TT$.
Then by inspection on the generators, we see that the representation $\pi_{\gamma_{z}}\times \gamma_{z}$ of  $\c\times_{\tau}^{\piso}\N$ on $\TT$ satisfies
$\pi_{\gamma_{z}}\times \gamma_{z}=\varepsilon_{z}\circ (\ell\times\delta^{*})$.
So the primitive ideal $J_{z}:=\ker \pi_{\gamma_{z}}\times \gamma_{z}$ of $\c\times_{\tau}^{\piso}\N$ is lifted from the quotient $(\c\times_{\tau}^{\piso}\N)/J \simeq C(\TT)$.

Since
$\pi_{\gamma_{z}}\times \gamma_{z}(f_{i,j}^{m})=0=\pi_{\gamma_{z}}\times \gamma_{z}(g_{i,j}^{m})$,
$\ker\pi_{\infty}=\ker\varphi_{T}$ and $\ker\pi_{\infty}^{*}=\ker\varphi_{T^{*}}$ are contained in $J_{z}$ for every $z\in\TT$.
Finally, since  $\pi_{\gamma_{z}}\times \gamma_{z}(v_{n+1})=\overline{z}^{n+1}\neq 0$ for $n\in \N$, $I_{n}\not\subset J_{z}$ for any $z\in\TT$.
\end{proof}

\begin{theorem}\label{set-prim}
The maps  $~n\in\N\cup \{\infty\}\cup \{\infty^{*}\} \mapsto I_{n}$ and $z\in\TT \mapsto J_{z}$ combine to give a bijection of the disjoint union
$\N\cup\{\infty\} \cup \{\infty^{*}\}\cup \TT$ onto $\Prim (\c\times_{\tau}^{\piso}\N)$, where $I_{\infty^{*}}:=\ker\varphi_{T}$.
Then the hull-kernel closure of a nonempty subset $F$ of

\[\N\cup\{\infty\} \cup \{\infty^{*}\}\cup \TT\]

is given by
\begin{itemize}
\item[(a)] the usual closure of $F$ in $\TT$ if $F\subset \TT$;
\item[(b)] $F$ if $F$ is a finite subset of $\N$;
\item[(c)] $F\cup\TT$ if $F \subset (\{\infty\} \cup \{\infty^{*}\})$;
\item[(d)] $F\cup (\{\infty\}\cup \{\infty^{*}\}\cup \TT)$ if $F\neq \N$ is an infinite subset of $\N$;
\item[(e)] $\N\cup \{\infty\}\cup \{\infty^{*}\}\cup \TT$ if $\N\subseteq F$.
\end{itemize}
\end{theorem}
\begin{proof}
The diagram \ref{diagram00} together with Proposition \ref{prim-a0} give a bijection map of
$\N\cup\{\infty\} \cup \{\infty^{*}\}\cup \TT$ onto $\Prim (\c\times_{\tau}^{\piso}\N)$.

Lemma \ref{property-In} (ii) gives the closure of the subset $F$ in (e), and Lemma \ref{property-In} (iii) gives the closure of the subset $F$ in (d).
If $F \subset (\{\infty\} \cup \{\infty^{*}\})$, then $\overline{F}=F\cup\TT$ because $\ker\pi_{\infty}^{*}, \ker\pi_{\infty}\subset J_{z}$ for every $z\in\TT$
by Lemma \ref{prim-CT}.

To see that $\overline{F}=F$ for a finite subset $F=\{n_{1},n_{2},\cdots,n_{j}\}$ of $\N$, we note that if an ideal $P\in\Prim (\c\times_{\tau}\N)$ satisfies
$\bigcap_{i=1}^{j}I_{n_{i}} \subset P$, then
\begin{itemize}
\item $P\neq J_{z}$ for any $z\in\TT$ because $v_{n_{j}+1}\in \bigcap_{i=1}^{j}I_{n_{i}}$ but $v_{n_{j}+1}\not\in J_{z}$ ;
\item $P\neq I_{\infty}, I_{\infty^{*}}$ because $v_{n_{j}+1}\in \bigcap_{i=1}^{j}I_{n_{i}}$ but $v_{n_{j}+1}\not\in I_{\infty}, I_{\infty^{*}}$;
\item $P\neq I_{n}$ for $n \not\in F$ because $(g_{0,0}^{n}-g_{0,0}^{n+1})\in \bigcap_{i=1}^{j}I_{n_{i}}$ but $(g_{0,0}^{n}-g_{0,0}^{n+1})\not\in I_{n}$ for $n \not\in F$.
\end{itemize}
So it can only be $P=I_{j}$ for some $j\in F$.
Finally the usual closure of $F$ in $\TT$ is followed by the fact that the map $z\mapsto J_{z}$ is a homeomorphism of $\TT$ onto the closed set $\Prim C(\TT)$.
\end{proof}

\end{document}